\documentclass[11pt, oneside]{article}    
\usepackage[english]{babel}
\usepackage[toc,page]{appendix}
\usepackage[margin=1in]{geometry}  
\usepackage{hyperref}
\usepackage{graphicx,wrapfig,lipsum}	
\usepackage[labelformat=empty]{caption}
\usepackage{yfonts}	
\usepackage{fixmath}
\usepackage{amssymb}
\usepackage{amsmath}
\usepackage{amsthm}
\usepackage{scalerel}
\usepackage{verbatim}
\newcommand{\Hplus}{+_{_H}}
\newcommand{\Vplus}{+_{_V}}

\newcommand{\Rows}{\mathsf{DistRows}}
\newcommand{\Hook}{\mathsf{Hook}}
\newcommand{\Rect}{\mathsf{Rect}}
\newcommand{\height}{\mathsf{height}}
\newcommand{\Irrep}{\mathsf{Irrep}}

\newcommand{\llambda}{\mathbold{\lambda}}
\newcommand{\mmu}{\mathbold{\mu}}
\newcommand{\nnu}{\mathbold{\nu}}
\newcommand{\ggamma}{\mathbold{\gamma}}
\newcommand{\cchi}{\mathbold{\chi}}
\newcommand{\vvarrho}{\mathbold{\varrho}}
\newcommand{\ttau}{\mathbold{\tau}}
\newcommand{\ttheta}{\mathbold{\theta}}
\newcommand{\Reg}{\mathbold{\textbf{Reg}}}

\usepackage{framed}

\DeclareMathOperator*{\Hsum}{\scalerel*{\enspace \Sigma_{_H}}{\sum}}
\DeclareMathOperator*{\Vsum}{\scalerel*{\enspace \Sigma_{_V}}{\sum}}

\usepackage[colorinlistoftodos]{todonotes}

\usepackage{ytableau}
\usepackage{thm-restate}

\usepackage{mathtools}

\newcommand{\verteq}{\rotatebox{90}{$\,=$}}
\newcommand{\equalto}[2]{\underset{\scriptstyle\overset{\mkern4mu\verteq}{#2}}{#1}}

\definecolor{lightgreen}{rgb}{0,0.8,0}
\definecolor{darkgreen}{rgb}{0,0.3,0}
\definecolor{lightblue}{rgb}{0,0,0.65}
\definecolor{darkblue}{rgb}{0,0,0.4}
\definecolor{lightred}{rgb}{0.8,0,0}
\definecolor{darkred}{rgb}{0.3,0,0}

\newtheorem{thm}{Theorem}[section]
\newtheorem{lem}[thm]{Lemma}
\newtheorem{prop}[thm]{Proposition}
\newtheorem{cor}[thm]{Corollary}
\newtheorem{conj}[thm]{Conjecture}
\theoremstyle{definition}
\newtheorem{defn}[thm]{Definition}
\newtheorem{rem}[thm]{Remark}

\ytableausetup{nosmalltableaux}
\ytableausetup{boxsize=1em}

\title{Covering $\Irrep(S_n)$ With Tensor Products and Powers}

\date{}		
\author{Mark Sellke}					
\begin{document}
\maketitle

\abstract{We study when a tensor product of irreducible representations of the symmetric group $S_n$ contains all irreducibles as subrepresentations --- we say such a tensor product \emph{covers} $\Irrep(S_n)$. Our results show that this behavior is typical. We first give a general sufficient criterion for tensor products to have this property, which holds asymptotically almost surely for constant-sized collections of (Plancherel or uniformly) random irreducibles. We also consider the minimal tensor power of a single fixed irreducible representation needed to cover $\Irrep(S_n)$. Here a simple lower bound comes from considering dimensions, and we show it is always tight up to a universal constant factor as was recently conjectured by Liebeck, Shalev, and Tiep.}

\vspace{0.5cm}

\tableofcontents

\newpage

\section{Introduction}

A vast amount is known about representations of the symmetric groups. However, additive decompositions of their tensor products into irreducibles have proven difficult to study. These decompositions are defined by \emph{Kronecker coefficients} which also appear in the study of geometric complexity theory (see \cite{GCT}) and quantum mixed states (see \cite{quantum1,quantum2}). Even checking whether a Kronecker coefficient vanishes is known to be NP hard \cite{IkenHard,IkenNP}. By contrast, the related Littlewood-Richardson coefficients and character values of irreducible representations have been long known to have combinatorial interpretations via the Littlewood-Richardson and Murnaghan-Nakayama rules. The Saxl Conjecture below encapsulates some of this lack of understanding.

\begin{defn}

We say a representation $V$ of the symmetric group $S_n$ \emph{covers} $\Irrep(S_n)$ if it contains all irreducible representations of $S_n$ as subrepresentations.

\end{defn}

\begin{conj}[Saxl Conjecture]

\label{conj:saxl}
For every $n$ except $2,4,9$ there exists an irreducible representation $\llambda$ of $S_n$ such that $\llambda\otimes \llambda$ covers $\Irrep(S_n)$. Furthermore when $n=\binom{r+1}{2}$ is a triangular number, we can take $\llambda$ to be the \emph{staircase representation} $\vvarrho_r$ corresponding to the Young diagram $(r,r-1,\dots,1)$.

\end{conj}

This conjecture, proposed in a 2012 lecture by Saxl, has attracted recent interest and parallels the work \cite{steinbergsquare} establishing an analogous result in groups of Lie type. The paper \cite{Pak} is the first to study Conjecture~\ref{conj:saxl} and shows that $\vvarrho_r^{\otimes 2}$ contains all hooks and two-row partitions. Moreover they conjecture that various other shapes should suffice in place of the staircase. The work \cite{IkenDom} shows that any Young diagram comparable to a staircase in the dominance partial order is contained in $\vvarrho_r^{\otimes 2}$. \cite{li2018} shows that irreducibles with Durfee size up to $3$ are contained, and reduces the same result for any fixed Durfee size to a finite case check.

In our previous work \cite{LuoSellke}, we showed $\vvarrho_r^{\otimes 2}$ contains asymptotically almost all irreducible representations in uniform and Plancherel measure, and that the tensor fourth power $\vvarrho_r^{\otimes 4}$ contains all irreducibles.\footnote{We show the latter from the former in Appendix A, simplifying \cite{LuoSellke} which used two separate arguments.} The methods of that paper are highly specialized to the staircase, relying on decomposing staircases into smaller staircases and the aforementioned result of \cite{IkenDom}. It is natural to wonder if such covering results hold more generically.

This paper shows that covering $\Irrep(S_n)$ is indeed a generic behavior for tensor products of irreducible representations. To the best of our knowledge, there have been no comparable prior results toward this rather natural question. We begin by giving a general sufficient criterion for such a tensor product to cover $\Irrep(S_n)$. Roughly, the set of tensored Young diagrams must contain enough \emph{disjoint similar pairs}. A primary application is the following corollary, which shows that a constant number of \emph{random} irreducible representations cover $\Irrep(S_n)$ when tensored. Here and throughout we use the partition notation $\llambda\vdash n$ to indicate that $\llambda$ is an irreducible representation of $S_n$, in accordance with the bijection between irreducible representations and partitions or Young diagrams.

\begin{restatable}{thm}{randomproduct}
\label{thm:randomproduct}
There exists an absolute constant $k\in\mathbb Z_+$ such that if $\llambda_{(1)},\dots,\llambda_{(k)}\vdash n$ are arbitrarily coupled Plancherel or uniformly random irreducible representations of $S_n$, then the tensor product \[\bigotimes_{i=1}^k \llambda_{(i)}\] covers $\Irrep(S_n)$ asymptotically almost surely, i.e. with probability $1-o_{n}(1)$.
\end{restatable}

Another interpretation of this result is as follows. Recall that the Kronecker coefficients $g\big(\llambda,\mmu,\nnu\big)$ are symmetric in their three arguments and defined by
\[
\llambda\otimes \mmu=\bigoplus_{\nnu}g(\llambda,\mmu,\nnu)\nnu.
\]
We may similarly define extended Kronecker coefficients $g\big(\llambda_{(1)},\llambda_{(2)},\dots,\llambda_{(k+1)}\big)$ so that 
\[
\bigotimes_{i=1}^k\llambda_{(i)} = \bigoplus g\big(\llambda_{(1)},\llambda_{(2)},\dots,\llambda_{(k+1)}\big) \llambda_{(k+1)}.
\]
The following corollary, immediate from Theorem~\ref{thm:randomproduct}, states that almost all extended Kronecker coefficients are positive for a suitable constant $k$. (In fact the discussion in Appendix A shows how to obtain the reverse implication from Corollary~\ref{cor:randomkron} to Theorem~\ref{thm:randomproduct}.) It would be very interesting to establish Corollary~\ref{cor:randomkron} for $k=2$, i.e. to show that almost all ordinary Kronecker coefficients are positive.

\begin{cor}
\label{cor:randomkron}
There exists an absolute constant $k\in\mathbb Z_+$  such that if $\llambda_{(1)},\dots,\llambda_{(k+1)}\vdash n$ are arbitrarily coupled Plancherel or uniformly random irreducible representations of $S_n$, then the extended Kronecker coefficient $g\big(\llambda_{(1)},\llambda_{(2)},\dots,\llambda_{(k+1)}\big)$ is nonzero asymptotically almost surely.
\end{cor}

We devote further attention to tensor powers of a single irreducible representation $\llambda$. In particular, we answer the question of what $t=t(\llambda)$ is necessary for $\llambda^{\otimes t}$ to cover $\Irrep(S_n)$. Because $S_n$ has irreducible representations of dimension $n^{\Omega(n)}$, it follows that $t\geq \Omega\left(\frac{n\log n}{\log\dim(\llambda)}\right)$ is necessary. In fact the same simple lower bound holds for any non-abelian simple group $G$ with $n\log n$ replaced by $\log |G|$; \cite{liebeck2019diameters} recently conjectured that this bound is sharp up to a universal constant factor for all irreducible representations of such $G$. \cite{liebeck2019diameters} proved this result for simple groups of Lie type and bounded rank, but described the case of the symmetric group $S_n$ as ``wide open". Our final theorem affirmatively resolves their conjecture for $S_n$.

\begin{restatable}{thm}{powersquantitative}
\label{thm:powersquantitative}

There exists an absolute constant $C$ such that for any $\llambda\vdash n$ {with $\dim(\llambda)>1$}, the tensor power $\llambda^{\otimes t}$ covers $\Irrep(S_n)$ for all $t\geq \frac{Cn\log n}{\log\dim(\llambda)}$. 

\end{restatable}

We give a few appealing corollaries to help interpret this result. The first follows by simple estimates using the hook-length formula (or see Lemma~\ref{lem:rectsquare}). It implies that a constant power of any Young diagram with non-degenerate limiting shape under constant-area rescaling requires $O(1)$ tensor powers to cover $\Irrep(S_n)$. The second is a uniform bound showing that $\llambda^{\otimes O(n)}$ always suffices --- it follows from the fact that all irreducible $S_n$ representations have dimension $1$ or at least $n-1$, and was established as Theorem 5 of \cite{liebeck2019diameters} as well. The third is a refinement showing that unless $\llambda$ is essentially a single row or column we obtain the much smaller bound $\llambda^{\otimes O(\log n)}$. It follows from the proof of Theorem~\ref{thm:powersquantitative}; in particular all $\llambda$ considered in cases 1 and 2 of the proof satisfy the statement, and in case 3 it follows from Corollary~\ref{cor:dimllambda}. The fourth extends Theorem~\ref{thm:powersquantitative} to multiplicity-free representations $V$, i.e. those containing each irreducible representation at most once. It follows from considering the largest irreducible subrepresentation of $V$. Indeed as explained in the introduction of \cite{liebeck2019diameters}, this extension is automatic in all simple groups, and $\Irrep(S_n),\Irrep(A_n)$ contain representations with the same dimensions up to factors of $2$ so the same argument works for $S_n$.

\begin{cor}

Fix $0<\varepsilon\leq\frac{1}{2}$. If $\llambda$ is an irreducible $S_n$ representation containing an $\Omega(n^{\varepsilon})\times \Omega(n^{1-\varepsilon})$ rectangle inside its associated Young diagram, then $\llambda^{\otimes O(1/\varepsilon)}$ covers $\Irrep(S_n)$.

\end{cor}

\begin{cor}[Theorem 5 of \cite{liebeck2019diameters}]

If $\llambda$ is an irreducible $S_n$ representation with $\dim(\llambda)>1$ then $\llambda^{\otimes O(n)}$ covers $\Irrep(S_n)$.

\end{cor}

\begin{cor}

If $\llambda$ has first row and first column each of length at most $n(1-\Omega(1))$, then $\llambda^{\otimes O(\log \, n)}$ covers $\Irrep(S_n)$.

\end{cor}

\begin{cor}

There exists an absolute constant $C$ such that for any multiplicity-free representation $V$ of $S_n$ with $\dim(V)\geq 3$, the tensor power $V^{\otimes t}$ covers $\Irrep(S_n)$ for all $t\geq \frac{Cn\log n}{\log\dim(V)}$. 

\end{cor}

Our proofs rely crucially on two results. The first is the previously mentioned Theorem~1.4 of \cite{LuoSellke} which states that $\vvarrho_r^{\otimes 4}$ covers $\Irrep(S_n)$. This may be surprising as the results of the present paper make no explicit mention of the staircase. The second is the \emph{semigroup property} of positive Kronecker coefficients which allows us to combine tensor information from $S_k,S_{\ell}$ into information on $S_{k+\ell}$. This was also the main tool of \cite{LuoSellke} as well as \cite{li2018}. A key idea is that Young diagrams with many distinct rows contain a staircase in an appropriate sense (see Proposition~\ref{prop:commonrows} for details). The semigroup property allows us to relate tensor products of these staircases to tensor products of the larger irreducible representations. Our use of the semigroup property is rather different from prior work, and the result of \cite{LuoSellke} essentially serves as a finishing step in our arguments. 

{
\begin{rem}
The conjecture of \cite{liebeck2019diameters} was for simple groups, which include the alternating groups $A_n$ but not $S_n$. Subsequent to our work, \cite{liebeck2021mckay} used Theorem~\ref{thm:powersquantitative} and induction/restriction arguments to deduce the same result for $A_n$. The extension is not immediate because self-conjugate $\llambda\vdash n$ split into pairs of irreducible $A_n$ representations.
\end{rem}
}

\section{Background}

In this section we give relevant definitions and prior results, largely overlapping with \cite{LuoSellke}. A notable new parameter is the number of distinct row lengths of a Young diagram, which did not appear in that work but is key here.

\subsection{Notations}

Throughout we use the terms ``irreducible representation", ``Young diagram", and ``partition" essentially interchangeably. We mean by $\llambda=(a_1,\dots,a_k)\vdash n$ that $\llambda$ has row lengths $a_1\geq a_2\geq\dots\geq a_k\geq 0$ summing to $n$, and also write $|\llambda|=n$. We denote by $\mathcal Y_n$ the set of Young diagrams with $n$ boxes. We denote by $\llambda'$ the conjugate Young diagram of $\llambda$. We denote by $1_n$ the trivial representation or horizontal strip and $1^n$ the alternating representation or vertical strip. We denote by $\vvarrho_r=(r,r-1,\dots,1)\vdash \binom{r+1}{2}$ the staircase and $\ttau_n=(n)\oplus (n-1,1)$ the (reducible) \emph{standard representation}. We set $\Rect(a,b)=(a,a,\dots,a)\vdash ab$ to be the rectangle with $a$ columns and $b$ rows and $\Hook(a,b)=(a,1,1,\dots,1)\vdash a+b-1$ the diagram with a row of length $a$ and column of length $b$. 

We will consider two probability measures on $\mathcal Y_n$: the self-explanatory \emph{uniform measure} and the algebraically natural \emph{Plancherel measure} which assigns probability $\frac{(\dim\llambda)^2}{n!}$ to each $\llambda\vdash n$. 

\begin{figure}[h]
\begin{framed}
\begin{align*}
1_5&=\ydiagram{5}\quad\quad\quad1^3=\ydiagram{1,1,1}\quad& \ttau_4=\ydiagram{3,1}\oplus\ydiagram{4}  \\
\vvarrho_5&=\ydiagram{5,4,3,2,1}\quad &\Rect(5,4)=\ydiagram{5,5,5,5}\\
\Hook(7,4)&=\ydiagram{7,1,1,1}\quad &(10,6,4)=\ydiagram{10,6,4}
\end{align*}

\caption{Illustrations of our notations for various $S_n$ representations.}
\end{framed}
\end{figure}

If $V$ is a possibly reducible $S_n$-representation we will often identify $V$ with its set of irreducible subrepresentations and write $\llambda\in V$ or even $W\subseteq V$ accordingly.

\subsection{Kronecker Coefficients and Relations}

We recall the \emph{Kronecker coefficients} $g\big(\llambda,\mmu,\nnu\big)\geq 0$ which are given for $\llambda,\mmu,\nnu\vdash n$ by
{
\[
	g(\llambda,\mmu,\nnu)=\langle \chi^{\llambda},\chi^{\mmu}\chi^{\nnu}\rangle.
\]
Here $\chi$ denotes the character; note that $g\big(\llambda,\mmu,\nnu\big)$ is symmetric in its $3$ arguments.
} 
By the nature of our results, we are only concerned with whether certain Kronecker coefficients vanish rather than their actual values. Therefore as in \cite{LuoSellke} we adopt the following notation to indicate that a Kronecker coefficient is positive.

\begin{defn}

Let $c\big(\llambda,\mmu,\nnu\big)$ denote the statement that the Kronecker coefficient $g\big(\llambda,\mmu,\nnu\big)$ is positive, or equivalently that $\llambda\otimes \mmu\otimes \nnu$ contains the trivial representation $1_n$ in its direct sum decomposition. More generally, let $c\big(\llambda_{(1)},\llambda_{(2)},\dots,\llambda_{(k)}\big)$ denote the statement that $\bigotimes_{i=1}^k \llambda_{(i)}$ contains $1_n$ in its direct sum decomposition, i.e. that the corresponding extended Kronecker coefficient (as defined in the introduction) is positive.

\end{defn}

We call a statement $c\big(\llambda_{(1)},\llambda_{(2)},\dots,\llambda_{(k)}\big)$ a \emph{Kronecker relation}. Identifying arbitrary representations with subsets of irreducibles, we may equivalently write $c\big(\llambda,\mmu,\nnu\big)$ as $\llambda\in \mmu\otimes \nnu$. In following subsections we will give general criteria for Kronecker relations to hold. For now we point out the simple but crucial fact that overlapping Kronecker relations can be combined. For instance, the three relations
\[
c\big(\llambda_{(1)},\llambda_{(2)},\llambda_{(3)}\big)
,\quad 
c\big(\mmu_{(1)},\mmu_{(2)},\mmu_{(3)}\big)
,\quad 
c\big(\llambda_{(3)},\mmu_{(3)},\nnu\big)
\] 
imply the further relation 
\[
c\big(\llambda_{(1)},\llambda_{(2)},\mmu_{(1)},\mmu_{(2)},\nnu\big).
\] 
When we wish to emphasize such a step we may phrase the first three relations as 
\[
\llambda_{(3)}\in \llambda_{(1)}\otimes \llambda_{(2)},\quad \mmu_{(3)}\in  \mmu_{(1)}\otimes \mmu_{(2)},\quad \nnu\in \llambda_{(3)}\otimes \mmu_{(3)}
\]
so that the implication
\[
\nnu\in \llambda_{(1)}\otimes\llambda_{(2)}\otimes\mmu_{(1)}\otimes\mmu_{(2)}
\]
follows simply because $V\subseteq \tilde V, W\subseteq\tilde W\implies V\otimes W\subseteq \tilde V\otimes \tilde W$ for any (possibly reducible) representations $V,W,\tilde V,\tilde W$.

\subsection{Criteria for Kronecker Relations}

As in \cite{LuoSellke}, we make crucial use of the $\textit{semigroup property}$ of \cite{semigp}. To state this, we first define the \textit{horizontal sum} of partitions. This operation adds row lengths, or equivalently forms the disjoint union of the column-length multisets.

\begin{defn}
The \textit{horizontal sum} $\llambda_{(1)}\Hplus\llambda_{(2)}$ of partitions $\llambda_{(1)}=(a_1,a_2,\dots ) \vdash n_1$ and $\llambda_{(2)} =(b_1, b_2 \dots) \vdash n_2$ is $\llambda_{(1)}\Hplus\llambda_{(2)}:=(a_1+b_1, a_2+b_2,  \dots ) \vdash (n_1+n_2)$. We denote longer sums by
\[
\Hsum_{i=1}^k \llambda_{(i)} = \llambda_{(1)}\Hplus\llambda_{(2)}\Hplus\dots\Hplus\llambda_{(k)}.
\]
This operation is well-defined because $\Hplus$ is commutative and associative.
\end{defn}

Vertical addition is denoted by $\Vplus,\Vsum$ and defined analogously.
\begin{defn}
The \textit{vertical sum} $\llambda_{(1)}\Vplus\llambda_{(2)}$ of $\llambda_{(1)},\llambda_{(2)}$ is the partition formed by unioning the row-length multisets of $\llambda_{(1)},\llambda_{(2)}$. Equivalently, $\llambda_{(1)}\Vplus\llambda_{(2)}=(\llambda_{(1)}'\Hplus\llambda_{(2)}')'$. Similarly define $\Vsum$.
\end{defn}

\begin{figure}[h]
\begin{framed}
\noindent\begin{minipage}{.5\linewidth}
\[\ydiagram{3,2}\Hplus \ydiagram{4,2,1} = \ydiagram{7,4,1}.\]
\end{minipage}%
\begin{minipage}{.5\linewidth}
\begin{align*}&\ydiagram{3,2}\\ 
&\quad\quad\Vplus\\
&\equalto{\ydiagram{4}}{\ydiagram{4,3,2}}
\end{align*}
\end{minipage}

\caption{Examples of horizontal and vertical sums. Horizontal summation can be defined by adding row lengths, or equivalently by unioning column-length multisets. }
\end{framed}

\end{figure}

We now state the \textit{semigroup property}. The $k=3$ case was proved in \cite{semigp} and it immediately implies the version for larger $k$ as observed in \cite{LuoSellke}.

{
\begin{thm} [Semigroup Property, {\cite[Theorem~3.1]{semigp},\cite[Lemma~C.1]{LuoSellke}}]
\label{thm:Hsum}
If both 
\[
	c\big(\llambda_{(1)}, \llambda_{(2)}, \llambda_{(3)},\dots,\llambda_{(k)}\big),\quad 
	c\big(\mmu_{(1)},\mmu_{(2)}, \mmu_{(3)},\dots,\mmu_{(k)}\big)
\]
hold then we also have
\[
	c\big(\llambda_{(1)}\Hplus\mmu_{(1)}, \llambda_{(2)}\Hplus\mmu_{(2)},\dots, \llambda_{(k)} \Hplus \mmu_{(k)}\big).
\] 

\end{thm}

\begin{cor} \label{cor:vsum}
If both 
\[
	c\big(\llambda_{(1)}, \llambda_{(2)}, \llambda_{(3)},\dots,\llambda_{(k)}\big),\quad 
	c\big(\mmu_{(1)},\mmu_{(2)}, \mmu_{(3)},\dots,\mmu_{(k)}\big)
\]
hold then we also have
\[
	c\big(\llambda_{(1)}\Vplus\mmu_{(1)},\dots \llambda_{(2j)}\Vplus\mmu_{(2j)},\llambda_{(2j+1)}\Hplus\mmu_{(2j+1)},\dots, \llambda_{(k)} \Hplus \mmu_{(k)}\big).
\]
\end{cor}
}

\begin{proof}

Since conjugating pairs does not affect Kronecker coefficients we have 
\[
	c\big(\llambda'_{(1)},\dots,\llambda'_{(2j)},\llambda_{(2j+1)},\dots,\llambda_{(k)}\big),\quad
	c\big(\mmu'_{(1)},\dots,\mmu'_{(2j)},\mmu_{(2j+1)},\dots,\mmu_{(k)}\big).
\]
Applying Theorem~\ref{thm:Hsum} and conjugating the first $2j$ entries back yields the result.
\end{proof}

Corollary~\ref{cor:vsum} states that in using the semigroup property we are allowed to use an even number of vertical additions in each step. It is \textit{not} true that vertically adding all 3 partitions preserves constituency. For example, $c\big((1),(1),(1)\big)$ holds for the trivial representation of $S_1$, but vertically adding this to itself would give a false statement since the alternating representation of $S_2$ is not contained in its own tensor square.

\begin{figure}

\begin{framed}

\[c\begin{pmatrix}\ydiagram{3,2,1}\\ \\ \ydiagram{6,0} \\ \\ \ydiagram{3,2,1}\end{pmatrix}\Hplus c\begin{pmatrix}\ydiagram{4}\\ \\ \ydiagram{1,1,1,1} \\ \\ \ydiagram{1,1,1,1}\end{pmatrix}=c\begin{pmatrix}\ydiagram{7,2,1}\\ \\ \ydiagram{7,1,1,1}\\ \\ \ydiagram{4,3,2,1}\end{pmatrix}\]

\caption{An example of the semigroup property. Horizontally summing the Young diagrams in the first two Kronecker relations yields a third Kronecker relation which is far less apparent. }
\end{framed}
\end{figure}

We will sometimes use the basic result that the tensor product $\llambda\otimes \ttau_n$ of $\llambda$ with the standard representation contains exactly the Young diagrams within blockwise distance $1$ of $\llambda$. Finally we mention an interesting and useful result from \cite{TensorCube}.

\begin{lem}[\cite{TensorCube}]
\label{lem:tensorcube} 
If $\llambda=\llambda'$ is symmetric then $c\big(\llambda,\llambda,\llambda\big)$ holds.
\end{lem}

{
\subsection{Blockwise Distance}

Define the \emph{blockwise distance} $d(\llambda,\mmu)$ between $\llambda\vdash n$ and $\mmu\vdash n$ to be the smallest number of squares which must be moved from $\llambda$ to transform it into $\mmu$. It will be useful that block-wise distance is subadditive under horizontal summation:
\[
	d(\llambda_{(1)}\Hplus \llambda_{(2)},\mmu_{(1)}\Hplus\mmu_{(2)})\leq d(\llambda_{(1)},\mmu_{(1)})+d(\llambda_{(2)},\mmu_{(2)}).
\]
Further, we call $\frac{d(\llambda,\mmu)}{n}$ the \emph{rescaled blockwise distance}. The original blockwise distance is equivalent to the $L^1$ norm on the indicator functions of the Young diagrams viewed as a subset of the plane, and the rescaled block distance corresponds to the same norm upon dilating each $\mathcal Y_n$ by a factor of $n^{-1/2}$. Using this equivalent definition, we extend the rescaled blockwise distance to \emph{continuous} Young diagrams, i.e. negative, increasing cadlag functions $f:\mathbb R_+\to \mathbb R_-$ with total integral $-1$. See \cite{LuoSellke} for a longer discussion of continuous Young diagrams.
}

\subsection{Numbers of Distinct Row Lengths}

 Some of our results require that Young diagrams contain many distinct row lengths, or even many distinct shared row lengths. Here we give two simple results on this statistic.

\begin{defn}

Let $\Rows(\llambda)$ be the number of distinct row lengths of $\llambda$, and $\Rows(\llambda,\mmu)$ the number of \textbf{shared} distinct row lengths of $\llambda,\mmu$.

\end{defn}

\begin{prop}

Any partition $\llambda$ has exactly $\Rows(\llambda)$ distinct column lengths. Furthermore $\Rows(\llambda\Hplus\mmu)\geq \Rows(\llambda)$.

\end{prop}

\begin{proof}

The first part is because the boundary of the shape of $\llambda$ consists of alternating horizontal and vertical line segments. The second is because horizontal summation is equivalent to union of column-length multisets.
\end{proof}

\begin{prop}
\label{prop:commonrows}

The partition $\llambda$ satisfies $\Rows(\llambda)\geq r$ if and only if it can be written as
\[
\llambda=\mmu\Vplus (\nnu\Hplus \vvarrho_r)
\]
for suitable partitions $\mmu,\nnu$. Furthermore the partitions $\big(\llambda_{(1)},\llambda_{(2)},\dots,\llambda_{(k)}\big)$, possibly of differing sizes, share at least $r$ distinct row lengths if and only if there exist partitions $(\mmu_{(1)},\dots,\mmu_{(k)},\nnu)$ such that
\[
\llambda_{(j)}=\mmu_{(j)}\Vplus (\nnu\Hplus \vvarrho_r),\quad \forall~ 1\leq j\leq k.
\]
\end{prop}

\begin{proof}

Because $\vvarrho_r$ has $r$ distinct row lengths, the previous proposition implies that $\mmu\Vplus (\nnu\Hplus \vvarrho_r)$ does as well. Conversely, given $\llambda$ with $r$ distinct row lengths, hence column lengths, we can take $(\nnu\Hplus \vvarrho_r)$ to consist of those $r$ rows; since their length are distinct it is easy to see that some suitable $\nnu$ exists. We take $\mmu$ to consist of the remaining rows in $\llambda$, so that vertical summation combines the row multisets to give $\llambda$. The second part is similar.
\end{proof}

\begin{figure}
\begin{framed}
\ytableausetup{boxsize=1em}
\begin{align*}
&\begin{matrix}\ydiagram [*(blue)]{0,4+3,0,1+2,0,1} *[*(red)]{0,7,0,3,0,1} *[*(green)]{8,0,4,0,3,0}\end{matrix}&= \begin{pmatrix} \ydiagram[*(green)]{8,4,3}\\ \Vplus \\ \ydiagram [*(blue)]{4+3,1+2,1} *[*(red)]{7,3,1}\end{pmatrix}&= \begin{pmatrix} \ydiagram[*(green)]{8,4,3}\\ \Vplus \\ \\  \ydiagram[*(red)]{4,1} \Hplus\ydiagram[*(blue)]{3,2,1}  \end{pmatrix} \\
&\begin{matrix}\ydiagram [*(blue)]{0,4+3,0,1+2,0,1} *[*(red)]{0,7,0,3,0,1} *[*(green)]{7,0,5,0,3,0}\end{matrix}&=\begin{pmatrix} \ydiagram[*(green)]{7,5,3}\\ \Vplus \\ \ydiagram [*(blue)]{4+3,1+2,1} *[*(red)]{7,3,1} \end{pmatrix} &=\begin{pmatrix} \ydiagram[*(green)]{7,5,3}\\ \Vplus \\ \\  \ydiagram[*(red)]{4,1} \Hplus \ydiagram[*(blue)]{3,2,1} \end{pmatrix} 
\end{align*}

\caption{An illustration of Proposition~\ref{prop:commonrows}, which shows how to extract a common staircase ``inside" partitions with shared row lengths. In this case the partitions $(8,7,4,3,3,1),(7,7,5,3,3,1)$ share distinct row lengths $(7,3,1)$ which are separated in the lower diagrams. Colors indicate how the diagrams on the right side are combined.} 
\end{framed}
\end{figure}

\subsection{The Fourth Power Saxl Theorem}

As explained in the introduction, the Saxl Conjecture asserts that $\vvarrho_r^{\otimes 2}$ contains all partitions of size $\binom{r+1}{2}$. Though the Saxl Conjecture is still open, the following fourth power Saxl theorem from \cite{LuoSellke} is just as good for our purposes.

\begin{thm}[Fourth Power Saxl Theorem; \cite{LuoSellke}, Theorem 1.4]
\label{thm:SaxlFourth}

For $r$ sufficiently large, the tensor fourth power $\varrho_r^{\otimes 4}$ covers $\Irrep(S_{\binom{r+1}{2}})$.

\end{thm}

For convenience we recall the following simple result which ensures that the sufficiently large condition above will not affect the results of this paper. 

\begin{prop}
\label{prop:smallpowers}
For any $\llambda\vdash n$, if $\dim(\llambda)>1$ then $\llambda^{\otimes t}$ covers $\Irrep(S_n)$ for some $t=t(\llambda)$.
\end{prop}

\begin{proof}
$A_n$ is simple, so the only non-faithful irreducible representations of $S_n$ are the trivial and alternating representations which are dimension $1$. Hence $\llambda$ is faithful. It is well-known that a large tensor power of any faithful representation covers $\Irrep(G)$ for any finite group $G$.
\end{proof}

\begin{rem}
{
Subsequent to the initial posting of this paper, \cite{harman2022tensor} improved Theorem~\ref{thm:SaxlFourth}, showing that the tensor cube $\varrho_r^{\otimes 3}$ covers $\Irrep(S_{\binom{r+1}{2}})$ and with no requirement that $r$ is sufficiently large. The latter point means that $\alpha_r\leq 2$ for all $r$ in Lemma~\ref{lem:2pair}, which presumably improves the implicit constant factors in our results.
}
\end{rem}

\section{Statements of Results}

Our first main result is a sufficient criterion for a tensor product to cover $\Irrep(S_n)$. The condition requires that the Young diagrams being tensored can be grouped into pairs so that each pair shares many distinct row lengths and has small blockwise distance.

\begin{thm}
\label{thm:criterion}

There exists an absolute constant $C$ such that the following holds. For a positive integer $n$, let $r\geq 2$ and $k\geq \frac{Cn}{r^2}$. Let $\llambda_{(1)},\dots,\llambda_{(k)},{\widetilde\llambda}_{(1)},\dots,{\widetilde\llambda}_{(k)}\vdash n$ and suppose that 
\begin{equation}
\label{eq:criterion-hypothesis}
\Rows(\llambda_{(i)},{\widetilde\llambda}_{(i)})\geq r
\quad\text{ and }\quad
d(\llambda_{(i)},{\widetilde\llambda}_{(i)})\leq \frac{r^2}{96}
\end{equation}
for all $1\leq i\leq k$. Then the tensor product
\[
	\bigotimes_{i=1}^{k} (\llambda_{(i)}\otimes {\widetilde\llambda}_{(i)})
\]
covers $\Irrep(S_n)$.
\end{thm}

Our first corollary is an immediate specialization to tensor powers.

\begin{cor}
\label{cor:powersqualitative}

Suppose $\Rows(\llambda)\geq r\geq 2$. Then 
\[
\llambda^{\otimes O(n/r^2)}
\] 
covers $\Irrep(S_n)$. Moreover the bound $O(n/r^2)$ is best possible.

\end{cor}

\begin{proof}

The first assertion is immediate from Theorem~\ref{thm:criterion}. To see why the bound is best possible note that $\vvarrho_r\Hplus 1_{n-\binom{r+1}{2}}\in \ttau_n^{\otimes O(r^2)}$ and therefore $\left(\vvarrho_r\Hplus 1_{n-\binom{r+1}{2}}\right)^{\otimes t}\subseteq\ttau_n^{\otimes O(tr^2)}$. Finally observe that $\ttau_n^{\otimes k}$ cannot cover $\Irrep(S_n)$ for $k\leq n-2$ because $d(1^n,1_n)=n-1$. This shows that $n/r^2$ tensor powers may be necessary. Moreover the trivial representation shows we must have $r\geq 2$ for the result to hold.
\end{proof}

The next corollary gives a covering criterion which only needs to be checked on the individual diagrams. It requires the existence of a suitable matching between enough pairs $\llambda_{(i)},\widetilde\llambda_{(i)}$ of Young diagrams. Given this, covering happens if each individual diagram has both $\Omega(\sqrt{n})$ distinct row lengths and a reasonably nondegenerate shape after dilation by $n^{-1/2}$. This comes at the cost of good quantitative dependence on the number of distinct rows.

\begin{cor}
\label{cor:manyrowssuffices}

Suppose $M_n$ is a sequence of probability measures on $\mathcal Y_n$ such that $\llambda$ sampled from $M_n$ asymptotically almost surely has $\Rows(\llambda)\geq \varepsilon\sqrt{n}$ for some fixed $\varepsilon>0$. Moreover suppose that at least one of the following two conditions holds:

\begin{enumerate}

\item A sample $\llambda$ from $M_n$ asympotically almost surely has all row and column lengths at most $C\sqrt{n}$, for some fixed constant $C$. 
\item $M_n$ converges in probability to an area $1$ limit shape with respect to rescaled blockwise distance.

\end{enumerate}

Then for some $k=k(C,\varepsilon)$ sufficiently large, if $\llambda_{(1)},\dots,\llambda_{(k)}$ are arbitrarily coupled samples from $M_n$, the tensor product $\bigotimes_{i=1}^k \llambda_{(i)}$ asympototically almost surely covers $\Irrep(S_n)$. (If the second condition holds then $k=k(\varepsilon)$ depends only on $\varepsilon$.)

\end{cor}

\begin{rem}

The two conditions of Corollary~\ref{cor:manyrowssuffices} regarding the shape of a typical $\llambda$ can be generalized to the following requirement: there exists a function $\delta(C)$ tending to $0$ such that asymptotically almost surely, at most $\delta n$ squares of a random $\llambda$ sampled from $M_n$ fall outside the upper-left $C\sqrt{n}\times C\sqrt{n}$ box. This requires essentially no change in the proof, and in fact $\delta$ only needs to eventually be below a small constant times $\varepsilon^2$. However we feel the given statement captures almost all interesting cases without unnecessary complication.

\end{rem}

Corollary~\ref{cor:manyrowssuffices} implies that tensor products of a constant number of uniformly or Plancherel random irreducible representations cover $\Irrep(S_n)$. All conditions of Corollary~\ref{cor:manyrowssuffices} are previously known, except the distinct rows condition for Plancherel measure which we verify using the results of \cite{Borodin}.

\randomproduct*

Finally we take a closer look at tensor powers and in particular the minimal $t=t(\llambda)$ such that $\llambda^{\otimes t}$ covers $\Irrep(S_n)$. Corollary~\ref{cor:powersqualitative} above gives an upper bound which is tight for some $\llambda$, but it is far from optimal for other shapes such as near-rectangles. As a trivial lower bound, observe that $t(\llambda)\geq \Omega(\frac{n\log n}{\log\dim(\llambda)})$ holds because there exist irreducible representations with dimension $n^{\Omega(n)}$. We show this bound is tight up to an absolute constant factor for all $\llambda$.

\powersquantitative*

The constants $k$ and $C$ in Theorems~\ref{thm:randomproduct} and~\ref{thm:powersquantitative} are both effective. We expect that for purposes of quantitative estimates, a better estimate of $k$ in Theorem~\ref{thm:randomproduct} could be obtained by verifying the distinct shared row lengths condition directly instead of using Corollary~\ref{cor:manyrowssuffices}.

\section{Proof of Theorem~\ref{thm:criterion} and Corollaries}

In this section we prove Theorem~\ref{thm:criterion} and then its corollaries. As mentioned in the introduction, the strategy is to apply Theorem~\ref{thm:SaxlFourth} to the staircases inside the Young diagrams (in the sense of Proposition~\ref{prop:commonrows}) and combine this information via the semigroup property.

\subsection{Proof of Theorem~\ref{thm:criterion}}

\begin{lem}
\label{lem:neartensor}

Suppose $\llambda,\widetilde\llambda\vdash n$ have blockwise distance $d(\llambda,\widetilde\llambda)\leq d$. Then there exists $\hat\ttheta\vdash n$ such that $d(1_n,\hat\ttheta)\leq d$ and $c\big(\llambda,\widetilde\llambda,\hat\ttheta\big)$ holds. Further, if $d(1_n,\hat\ttheta)\leq d$ holds for some $\hat\ttheta\vdash n$ with $n\geq 2d$, then there exists $\ttheta\vdash 2d$ such that $\hat\ttheta=1_{n-2d}\Hplus\ttheta$.
\end{lem}

\begin{proof}

Letting $\ttau_n$ denote the standard representation, we know that $\llambda\otimes {\ttau_n^{\otimes d}}$ contains $\widetilde\llambda$. This can be restated as
\[
c\big(\llambda\otimes {\ttau_n^{\otimes d}},\widetilde\llambda,1_n\big)
\]
which implies 
\[
	c\big(\llambda,\widetilde\llambda,{\ttau_n^{\otimes d}}\big).
\]
This implies the first statement since any subrepresentation of $\ttau_n^{\otimes d}$ has blockwise distance at most $d$ from $1_n$. For the second statement, it suffices to observe that the first row of $\hat\ttheta$ is at least $n-2d\geq 0$ longer than its second row.
\end{proof}

\begin{lem}
\label{lem:nearsharedrows}

Suppose $\llambda,\widetilde\llambda\vdash n$ have blockwise distance $d(\llambda,\widetilde\llambda)\leq d$ and $\Rows(\llambda,\widetilde\llambda)\geq r$, and {assume $n\geq 2d+\binom{r+1}{2}$}. Then there exists $\ttheta\vdash 2d$ such that any $\nnu\in\vvarrho_r^{\otimes 2}$ satisfies
\[
c\big(\llambda,\widetilde\llambda,{1_{n-2d-\binom{r+1}{2}}}\Hplus\ttheta\Hplus \nnu\big).
\]
\end{lem}

\begin{proof}

Let $\mmu$ be the partition with $r$ distinct-size rows (and no other rows), one for each shared row length of $\llambda,\widetilde\llambda$. From the decomposition of Proposition~\ref{prop:commonrows}, we can write
\begin{equation}
\label{eq:llambda=sum}
\llambda=\cchi\Vplus \mmu,\quad \widetilde\llambda=\widetilde\cchi\Vplus\mmu
\end{equation}
for $\cchi,\widetilde\cchi\vdash (n-|\mmu|)$ satisfying $d(\cchi,\widetilde\cchi)\leq d$. {By the first part of Lemma~\ref{lem:neartensor}, there exists $\hat\ttheta\vdash (n-|\mmu|)$ such that $d(\hat\ttheta,1_{n-|\mmu|})\leq d$} and 
\begin{equation}
\label{eq:cchi}
	c\big(\cchi,\widetilde\cchi, {\hat\ttheta}\big).
\end{equation}
Writing $\mmu= \ggamma\Hplus\vvarrho_r$ and applying the semigroup property to $c\big(\ggamma,\ggamma,1_{|\ggamma|}\big)$ and $c\big(\vvarrho_r,\vvarrho_r,\nnu\big)$ implies
\begin{equation}
\label{eq:mmu}
	c\big(\mmu,\mmu,{1_{|\ggamma|}\Hplus \nnu} \big)
\end{equation}
for any $\nnu\in\vvarrho_r^{\otimes 2}$. Combining \eqref{eq:cchi}, \eqref{eq:mmu} via the semigroup property with vertical summation in the first two components and recalling \eqref{eq:llambda=sum}, we find
\[
{
	c\big(\llambda,\widetilde\llambda,1_{|\ggamma|}\Hplus \hat\ttheta\Hplus \nnu\big)
	.
}
\]
{
By assumption, $|\ggamma|+|\hat\ttheta|=n-|\nnu|=n-\binom{r+1}{2}\geq 2d$. It is moreover easy to see that
\[
	d\big(1_{|\ggamma|}\Hplus \hat\ttheta,1_{|\ggamma|+|\hat\ttheta|}\big)\leq d(\hat\ttheta,1_{|\hat\ttheta|})\leq d.
\]
The second part of Lemma~\ref{lem:neartensor} now implies that
\[
	1_{|\ggamma|}\Hplus \hat\ttheta
	=
	1_{n-2d-\binom{r+1}{2}}\Hplus \ttheta
\]
for some $\ttheta\vdash 2d$, completing the proof.
}
\end{proof}

For convenience we now set $\alpha_r$ to be the smallest number such that $\vvarrho_r^{\otimes 2\alpha_r}$ covers $\Irrep(S_n)$. We recall that $\alpha_r\leq 2$ for $r$ sufficiently large by Theorem~\ref{thm:SaxlFourth} and that $\alpha_r$ is uniformly bounded.

\begin{lem}
\label{lem:2pair}

Let {$r,s\geq 1$ and $d\geq 0$} satisfy $n\geq\binom{r+1}{2}+2d$ and { $s=\binom{r+1}{2}-2d$}. Suppose that for each $1\leq i\leq {2}\alpha_r$, the pair $\llambda_{(i)},\widetilde\llambda_{(i)}\vdash n$ has blockwise distance $d(\llambda_{(i)},\widetilde\llambda_{(i)})\leq d$ and satisfies $\Rows(\llambda_{(i)},\widetilde\llambda_{(i)})\geq r\geq 2$. Then
\[
\bigotimes_{i=1}^{2\alpha_r} (\llambda_{(i)}\otimes\widetilde\llambda_{(i)})
\] 
contains $1_{n-s}\Hplus \mmu$ for all $\mmu\vdash s$.
\end{lem}

\begin{proof}
By Lemma~\ref{lem:nearsharedrows}, for each $1\leq i\leq 2\alpha_r$ there exists $\ttheta_{(i)}\vdash 2d$ such that 
\begin{equation}
\label{eq:llambda_i}
c\big(\llambda_{(i)},\widetilde\llambda_{(i)},1_m\Hplus\ttheta_{(i)}\Hplus {\nnu_{(i)}}\big)
\end{equation}
for arbitrary ${\nnu_{(i)}}\vdash \binom{r+1}{2}$ contained in $\vvarrho_r^{\otimes 2}$. 

Next for any $\mmu_A\vdash \binom{r+1}{2}$, by definition of $\alpha_r$ there exist ${\nnu_{(i)}}\in \vvarrho_r^{\otimes 2}$ for $1\leq i\leq \alpha_r$ such that 
\[
	c\big(\nnu_{(1)},\dots,{\nnu_{(\alpha_r)}},{\mmu_{A}}\big).
\]
Let $\ttheta_{A}\vdash 2d$ be an arbitrary partition with $c\big({\ttheta_{(1)},\ttheta_{(2)},\dots,\ttheta_{(\alpha_r)}},{\ttheta_{A}}\big)$. {Letting $m=n-2d-\binom{r+1}{2}\geq 0$}, the semigroup property shows
\[
c\big(1_m\Hplus{\ttheta_{(1)}}\Hplus \nnu_{(1)},1_m\Hplus{\ttheta_{(2)}}\Hplus \nnu_{(2)},\dots, 1_m\Hplus {\ttheta_{(\alpha_r)}}\Hplus {\nnu_{(\alpha_r)}}, 1_m\Hplus\ttheta_{A}\Hplus \mmu_{A}\big).
\]
Combining this with \eqref{eq:llambda_i} implies
\[
c\big(\llambda_{(1)},\widetilde\llambda_{(1)},\dots,{\llambda_{(\alpha_r)}},{\widetilde\llambda_{(\alpha_r)}},1_m\Hplus\ttheta_{A}\Hplus \mmu_{A}\big).
\]
Similarly there exists $\ttheta_B\vdash 2d$ such that for all $\mmu_B\vdash \binom{r+1}{2}$,
\[
c\big({\llambda_{(\alpha_r+1)}},{\widetilde\llambda_{(\alpha_r+1)}},\dots,{\llambda_{(2\alpha_r)}},{\widetilde\llambda_{(2\alpha_r)}},1_m\Hplus\ttheta_{B}\Hplus \mmu_{B}\big).
\]
Now for arbitrary $\mmu\vdash s$, {recalling that $\binom{r+1}{2}=s+2d$ by assumption} we set 
\begin{align*}
\mmu_{A}&=\ttheta_{B}\Hplus \mmu,
\\
\mmu_{B}&=\ttheta_{A}\Hplus 1_{s}.
\end{align*}
Then
\[
	c\big(\llambda_{(1)},\widetilde\llambda_{(1)},\dots,{\llambda_{(\alpha_r)}},{\widetilde\llambda_{(\alpha_r)}},1_{m}\Hplus (\ttheta_{A}\Hplus \ttheta_{B})\Hplus \mmu\big),
\]
\[
	c\big({\llambda_{(\alpha_r+1)}},{\widetilde\llambda_{(\alpha_r+1)}},\dots,{\llambda_{(2\alpha_r)}},{\widetilde\llambda_{(2\alpha_r)}},1_{m}\Hplus (\ttheta_{A}\Hplus \ttheta_{B})\Hplus 1_{s}\big).
\]
Recalling that $c\big(\ttheta_{A}\Hplus \ttheta_{B},\ttheta_{A}\Hplus \ttheta_{B},1_{|\ttheta_{A}|+ |\ttheta_{B}|}\big)$, the semigroup property now implies
\[
	c\big(1_{m}\Hplus (\ttheta_{A}\Hplus \ttheta_{B})\Hplus \mmu,{1_{m}}\Hplus (\ttheta_{A}\Hplus \ttheta_{B})\Hplus 1_{s},1_{n-s}\Hplus\mmu\big)
\]
Combining the previous displays, we find that
\[
	c\big(\llambda_{(1)},\widetilde\llambda_{(1)},\dots,{\llambda_{(2\alpha_r)}},{\widetilde\llambda_{(2\alpha_r)}},1_{n-s}\Hplus \mmu\big)
\]
holds for arbitrary $\mmu\vdash s$ as desired.
\end{proof}

\begin{proof}[Proof of Theorem~\ref{thm:criterion}]

{
By Proposition~\ref{prop:smallpowers} and the pigeonhole principle, we may assume $n$ is at least a large constant. If $r(r+1)< n$, set $\tilde r=r$. Otherwise, set $\tilde r=\lceil\sqrt{n}/2\rceil\leq r$. In the latter case, $r^2\leq r(r+1)\leq 2n\leq 8\tilde r^2$ must hold for $\llambda_{(1)}\vdash n$ to satisfy $\Rows(\llambda_{(1)})\geq r$. Therefore in either case $\tilde r$ satisfies $\tilde r\leq r\leq \tilde r\sqrt{8}$ and $\tilde r(\tilde r+1)< n$.

Next choose 
\[
	s\in 
	\left\{
	\left\lceil\frac{1}{2}\binom{\tilde r+1}{2}\right\rceil
	,
	\left\lceil\frac{1}{2}\binom{\tilde r+1}{2}\right\rceil+1
	\right\}
\]
to have the same parity as $\binom{\tilde r+1}{2}$. It is easy to see that $3\leq s\leq \binom{\tilde r+1}{2}$ since $\tilde r\geq 2$. Then we set
\begin{equation}
\label{eq:d_r^2_64}
\begin{aligned}
	d&:=\frac{\binom{\tilde r+1}{2}-s}{2}
	\\
	&\geq \frac{\tilde r(\tilde r+1)-6}{8}
	\geq 
	\left\lfloor\frac{\tilde r^2}{12}\right\rfloor
	\\
	&\geq 
	\left\lfloor\frac{r^2}{96}\right\rfloor.
\end{aligned}
\end{equation}
Then $s=\binom{\tilde r+1}{2}-2d$ by definition of $d$, and $\binom{\tilde r+1}{2}+2d= \tilde r(\tilde r+1)-s < n$ holds as well. These estimates allow us to apply Lemma~\ref{lem:2pair} to $(\tilde r,s,d)$. In particular \eqref{eq:d_r^2_64} together with the assumption \eqref{eq:criterion-hypothesis} implies $d(\llambda_{(i)},\mmu_{(i)})\leq d$ for $1\leq i\leq k$. In tandem with the second part of Lemma~\ref{lem:neartensor}}, we conclude that
 all irreducibles within blockwise distance at most $\lfloor s/2\rfloor$ {of the trivial representation} are contained in the tensor product 
\begin{equation}
\label{eq:tensor-s/2}
	\bigotimes_{{i\in [2\alpha_r]}}
	\left(\llambda_{(i)}\otimes \widetilde\llambda_{(i)}\right).
\end{equation}
In particular such tensor products contain everything in $\ttau_n^{\otimes \lfloor s/2\rfloor}$. As $s\geq 3$ we have $\lfloor s/2\rfloor \geq s/3$. {Since $\ttau_n^{\otimes n}$ covers $\Irrep(S_n)$ and \eqref{eq:tensor-s/2} holds for any set of $2\alpha_r$ indices $i$, we deduce that
\[
	\bigotimes_{i=1}^{2\alpha_{\tilde r}\lceil\frac{3n}{s}\rceil}
	\left(\llambda_{(i)}\otimes \widetilde\llambda_{(i)}\right).
\]
Using $s\geq \tilde r^2/4$ in the second step and $r^2\leq 2n$ in the last, we have
\begin{align*}
	2\alpha_{\tilde r}\left\lceil\frac{3n}{s}\right\rceil
	&\leq 
	\alpha_{\tilde r}\left(\frac{6n}{s}+2\right)
	\\
	&\leq 
	\alpha_{\tilde r}\left(\frac{24n}{\tilde r^2}+2\right)
	\\
	&\leq 
	\alpha_{\tilde r}\left(\frac{192 n}{r^2}+2\right)
	\\
	&\leq 
	\frac{196\alpha_{\tilde r}n}{r^2}.
\end{align*}
This completes the proof since $\alpha_{\tilde r}$ is uniformly bounded.
}
\end{proof}

\subsection{Proof of Corollary~\ref{cor:manyrowssuffices}}

\begin{proof}[Proof of Corollary~\ref{cor:manyrowssuffices}]

First we show that in either case, for fixed $m$ and sufficiently large $k=k(m)$, asymptotically almost surely there exist $m$ Young diagrams $\llambda_{(j_1)},\dots,\llambda_{(j_m)}$ inside an common $\varepsilon/10$ rescaled-blockwise metric ball {(for some $1\leq j_1<j_2<\dots <j_m\leq k$)}. We assume all of $\llambda_{(1)},\dots,\llambda_{(k)}$ are inside a $C\times C$ rescaled-box or are within $\frac{\varepsilon}{100}$ rescaled blockwise distance of the limit shape; each condition holds asymptotically almost surely in the respective case. 

In the first case, the set of continuous Young diagrams contained in a $C\times C$ box is relatively compact in the Hausdorff metric (viewed as subsets of the plane). This implies relative compactness in the rescaled blockwise metric as well. Indeed Minkowski-summing a small $\delta$ ball in the plane with a continuous Young diagram contained in a $C\times C$ box increases the area of the continuous Young diagram by at least $\Omega(\delta)$ and at most $O(C\delta)$, so that the two metrics are equivalent on the set of rescaled Young diagrams confined to a $C\times C$ box.

Taking $k$ to be $2m$ times the $\varepsilon/20$-rescaled blockwise covering number (finite by the above discussion) now suffices for the first case. In the second case the claim is immediate as asymptotically almost surely, all partitions are within $\varepsilon/20$ of the limiting shape.

{
Next we \textbf{claim} that there exists a pair among these nearby Young diagrams with $\Omega_{\varepsilon}(\sqrt{n})$ shared row lengths, for $m=m(\varepsilon)$ large enough. By removing this pair and repeating, this will imply (after e.g. doubling $m$) that at least $m/4$ such disjoint pairs can be formed. Applying Theorem~\ref{thm:criterion} now completes the proof assuming this claim (note that it suffices to apply Theorem~\ref{thm:criterion} to any subset of the $\llambda_{(1)},\dots,\llambda_{(k)}$).}

To show the above claim, we first focus on the case where all Young diagrams are confined to a $C\times C$ box, so we set $N=C\sqrt{n}$ and think of each $\llambda_{(j_i)}$ as a subset of $[N]$ via its distinct row lengths. The general fact we need is: given $10\alpha^{-1}$ subsets of $[N]$ with size at least $\alpha N$, there exist two with intersection at least $\frac{\alpha^2N}{2}$. This is well known to follow from a simple averaging argument over pairs of subsets. And indeed this fact with $\alpha=\varepsilon/C$ and $N=C\sqrt{n}$ suffices to show that there are pairs of Young diagrams among our $m$ with $\Omega(\varepsilon^2 {N}/C^2)$ shared distinct rows as desired, as long as $m\geq \frac{100C}{\varepsilon}$.

In the second case, the finish is identical after observing that at least $\frac{\varepsilon \sqrt{n}}{100}$ of the distinct row lengths of any $\llambda$ satisfying $\Rows(\llambda)\geq \varepsilon\sqrt{n}$ must be at most $\frac{100\sqrt{n}}{\varepsilon}$. This allows us to repeat essentially the same argument as above using $C=\frac{100}{\varepsilon}$.
\end{proof}

\subsection{Proof of Theorem~\ref{thm:randomproduct}}

Here we prove Theorem~\ref{thm:randomproduct} by checking the distinct rows condition of Corollary~\ref{cor:manyrowssuffices} in the Plancherel case.

\begin{lem}
There exists an absolute constant $\alpha>0$ such that a Plancherel random partition $\llambda\vdash n$ satisfies $\Rows(\llambda)=(\alpha+o(1)) \sqrt{n}$ asymptotically almost surely. 
\end{lem}

\begin{proof}
We apply the 2nd moment method, relying on Theorem 3 in \cite{Borodin}. Following that paper we set 
\[
	D(\lambda)=D((a_1,a_2,\dots,a_{\ell}))=\{a_i-i\}_{i=1}^{\ell}\subseteq\mathbb Z
\]
{to be the descent set of $\llambda$.} We count distinct row lengths via the number of $i\in\mathbb Z$ with both $i\in D(\llambda)$ and $i+1\notin D(\llambda)$ (it is easy to see this gives an exact count up to additive error $1$ which we ignore). Their work shows that when $i=(a+o(1))\sqrt{n}$ for $a\in [-2,2]$ the event $I_i=\{i\in D(\llambda),i+1\notin D(\llambda)\}$ has probability
\[
\mathbb P[I_i]= \frac{\arccos(a/2)}{\pi}-\det\begin{pmatrix} \frac{\arccos(a/2)}{\pi} & \frac{\sin(\arccos(a/2))}{\pi} \\
 \frac{\sin(\arccos(a/2))}{\pi} &  \frac{\arccos(a/2)}{\pi}\end{pmatrix}+o(1) := V(a)+o(1).
 \]
 Moreover a simple consequence of the Baik-Deift-Johansson theorem \cite{BDJ} is that asympotically almost surely, $I_i=0$ for all $|i|\geq 2n^{1/2}+n^{1/4}$. It is easy to see that $V(a)>0$ when $a\in (-2,2)$. From the preceding discussion we have
\[
\left|\Rows(\llambda)- \widetilde\Rows(\llambda)\right| \leq n^{1/4}
\]
where we define $\widetilde\Rows(\llambda)$ by 
\[\widetilde\Rows(\llambda):=\sum_{{i}=-2\sqrt{n}}^{2\sqrt{n}} 1_{I_i}.
\]
Therefore we have
\[
\mathbb E[\widetilde\Rows(\llambda)] =\sum_{{i}=-2\sqrt{n}}^{2\sqrt{n}} \mathbb P[I_i]=\sqrt{n}\left(\int_{-2}^2 V(a)da \right)(1+o(1)).
\]
Moreover \cite{Borodin} shows that when $i-j=\omega(1)$ is unbounded, the events $I_i,I_j$ are asymptotically independent. This means by definition that for any $\delta>0$ there exist $N,k$ such that if $n\geq N$ and $|i-j|\geq k$, then $Cor(I_i,I_j)\leq \delta$. From this the variance of $\widetilde\Rows(\llambda)$ is easily seen to be sublinear:
\[
Var\left[\widetilde\Rows(\llambda)\right]=o(n).
\] 
Hence the Chebychev inequality proves the lemma for $\alpha=\int_{-2}^2 V(a)da$.
\end{proof}

\begin{proof}[Proof of Theorem~\ref{thm:randomproduct}]

It is known that both Plancherel and uniformly random Young diagrams each converge to a limit shape (\cite{LoganShepp,PlanchLimShape,UnifLimShape}). Uniformly random Young diagrams have $(\frac{\sqrt 6}{\pi}+o(1))\sqrt n$ distinct rows asymptotically almost surely (see \cite{wilf}, or \cite{UniformRows} for a central limit theorem). We just verified that Plancherel random Young diagams contain $\Omega(\sqrt n)$ distinct row lengths. Hence Corollary~\ref{cor:manyrowssuffices} applies, completing the proof.
\end{proof}

\section{Proof of Theorem~\ref{thm:powersquantitative}}

Here we prove Theorem~\ref{thm:powersquantitative}, showing that $\llambda^{\otimes t}$ covers $\Irrep(S_n)$ for $t\geq O\left(\frac{n\log n}{\log\dim(\llambda)}\right)$. Note that we may assume $n$ is sufficiently large throughout via Proposition~\ref{prop:smallpowers}. The proof is split into three cases:

\begin{enumerate}
	\item $\dim(\llambda)\geq K^n$ for a large constant $K$.
	\item $\dim(\llambda)\leq K^n$, and $\llambda$ has at least $2$ rows and/or columns with length $\Omega(n)$.
	\item $\dim(\llambda)\leq K^n$, and $\llambda$ has only $1$ row or column with length $\Omega(n)$.
\end{enumerate}

We will show that $\dim(\llambda)\leq K^n$ implies the existence of a row/column with length $\Omega(n)$, so that the above combined with conjugation symmetry cover all cases for large $n$. 

The main case is the first. Using the hooklength formula for $\dim(\llambda)$, we show that the Young diagram for $\llambda$ contains large subsets which are possibly at different height scales. We then use the semigroup property to transform these large subsets into single rectangles and then single squares at each scale, all while using a small number of tensor powers of $\llambda$. The tensor cube of a square contains a Young diagram with many distinct row lengths, and applying Corollary~\ref{cor:powersqualitative} to each one allows us to obtain control over an arbitrary Young diagram of appropriate size for each scale. Finally we show that horizontally summing these arbitrary Young diagrams results in a single Young diagram with many distinct row lengths and we again apply Corollary~\ref{cor:powersqualitative} to conclude. We remark that without combining scales, we would lose a logarithmic factor, as is typical in dyadic pigeonhole arguments. More precisely, if $\dim(\llambda)\approx n^{\varepsilon n}$ we would obtain a slightly suboptimal upper bound $O(\varepsilon^{-1}\log(\varepsilon^{-1}))$ whereas by using multiple scales we achieve the tight bound $\Theta(\varepsilon^{-1})$. 

Case 2 is relatively straightforward. Case 3 goes by breaking $\llambda$ into the horizontal sum of $\hat\llambda$ with a long horizontal strip, where $\hat\llambda$ has first row length equal to that of either its second row or longest column. A key step in case 3 is to apply one of the previous cases to $\hat\llambda$, which by construction cannot itself fall into case 3.

\subsection{Preparatory Lemmas}

Here we prove various lemmas, primarily for use in case 1. 

\begin{lem}
\label{lem:rectcube}

We have $c\big(\Rect(ab,a),\Rect(ab,a),\Rect(ab,a)\big)$ for any positive integers $a,b$. More generally we have $c\big(\Rect(ab+c,a)),\Rect(ab+c,a), \Rect(ab,a)\Hplus 1_{ac}\big)$.
\end{lem}

\begin{proof}
The case $b=1$ for the first part is immediate from Lemma~\ref{lem:tensorcube} since squares are symmetric. For larger $b$, we apply the horizontal semigroup property repeatedly. The second part follows from another horizontal sum with $c\big(\Rect(c,a),\Rect(c,a),1_{ac}\big)$.
 \end{proof}

\begin{lem}

\label{lem:rectsquare}

For any positive integers $(x,y,z)$, 
\[
c\big(\Rect(xyz,xz),\Rect(xyz,xz),\Rect(yz^2,x^2)\big).
\]

\end{lem}

\begin{proof}

Because $\Rect(x,x)$ is symmetric, we have 
\[
c\big(\Rect(x,x),\Rect(x,x),\Rect(1,x^2)\big).
\]
Horizontally summing this relation $yz$ times with itself gives $c\big(\Rect(xyz,x),\Rect(xyz,x),\Rect(yz,x^2)\big)$. Vertically summing the first two components and horizontally summing the last component $z$ times with each of themselves gives the lemma.
\end{proof}

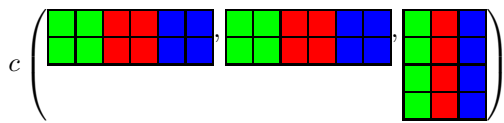
\begin{figure}[h]
\begin{framed}

\[c\begin{pmatrix}\ydiagram[*(green)]{2,2,0,0} *[*(red)]{2+2,2+2,0,0} *[*(blue)]{4+2,4+2,0,0},\ydiagram[*(green)]{2,2,0,0} *[*(red)]{2+2,2+2,0,0} *[*(blue)]{4+2,4+2,0,0},\ydiagram[*(green)]{1,1,1,1} *[*(red)]{1+1,1+1,1+1,1+1} *[*(blue)]{2+1,2+1,2+1,2+1}\end{pmatrix}\] 

\caption{In the main case $z=1$, Lemma~\ref{lem:rectsquare} states that the tensor square of a rectangle of height $x$ contains a rectangle of height $x^2$. Here we illustrate the case $(x,y,z)=(2,3,1)$. As usual the colors indicate how the semigroup property was applied.}
\end{framed}
\end{figure}
\begin{rem}

Lemma~\ref{lem:rectsquare} is in a sense the most efficient way to find Young diagrams with many rows in a tensor power of diagrams with few rows. The beautiful paper \cite{dvir} shows that the most total rows in any constituent of $\llambda\otimes \mmu$ exactly equals the number of blocks in the intersection of $\llambda$ and the transpose $\mmu'$, when they are overlayed with upper-left corners in the same location. In particular $c\big(\llambda,\mmu,\nnu\big)$ implies $\height(\nnu)\leq \height(\llambda)\cdot \height(\mmu)$, and Lemma~\ref{lem:rectsquare} is an equality case when $z=1$.

\end{rem}

The next lemma directly applies Lemma~\ref{lem:rectsquare} to show that a small number of tensor powers suffice to turn a rectangle into a square. We remark that in the main proof we round row and column lengths to powers of $16$ and not $2$ purely for the convenience of this lemma, which would otherwise have bothersome parity issues in the statement and proof. 

\begin{lem}
\label{lem:need16}
For any $b\geq a$ let $\llambda=\Rect(2^{4b},2^{4a})$. Then for all $t\geq O\left(\frac{b}{a}\right)$, the tensor power $\llambda^{\otimes 2t}$ contains $\Rect(2^{2a+2b},2^{2a+2b})$.

\end{lem}

\begin{proof}

We first remark that adding arbitrary even integers to a tensor exponent cannot hurt the statement, as $\llambda^{\otimes t}\subseteq \llambda^{\otimes t+2}$ holds for any $t\geq 0$. However we do need to take care that the exponent is even.

We apply Lemma~\ref{lem:rectsquare} with $z=1$ at most $\lfloor \log_2\left({\frac{b}{a}}\right)\rfloor$ times to repeatedly square the height of the rectangle: it shows $c\big(\Rect(2^{i+j},2^{i}),\Rect(2^{i+j},2^{i}),\Rect(2^{j},2^{2i})\big)$ for any non-negative integers $i,j$. Doing this as long as possible we obtain 
\begin{equation}
\label{eq:rect-recursion}
\Rect(2^{4v},2^{4u})\in \left(\Rect(2^{4b},2^{4a})^{\otimes \big(\frac{b}{a}+O(1)\big)}\right)^{\otimes 2} 
\end{equation}
for $(u,v)$ satisfying $u\leq v\leq 2u$ and $u+v=a+b$. Applying Lemma~\ref{lem:rectsquare} once more using 
\[
(x,y,z)=\left(2^{u+v},2^{4v-4u},2^{3u-v}\right)
\]
gives 
\[
c\big(\Rect(2^{4v},2^{4{u}}),\Rect(2^{4v},2^{4{u}}),\Rect(2^{2u+2v},2^{2u+2v})\big).
\]
{Since $\Rect(2^{2u+2v},2^{2u+2v})=\Rect(2^{2a+2b},2^{2a+2b})$, combining with \eqref{eq:rect-recursion} completes the proof.
}
\end{proof}

The next few lemmas show how to turn a square into an arbitrary Young diagram with a few more tensor powers.

\begin{lem}

\label{lem:squarecube}

$\Rect(2k,2k)^{\otimes 3}$ contains a Young diagram $\llambda$ with $\Rows(\llambda)=2k$.

\end{lem}

\begin{proof}

We have $(2,2),(4)\in \Rect(2,2)^{\otimes 2}\subseteq \Rect(2,2)^{\otimes 3}$. Therefore we have 
\[
c\big(\Rect(2k,2),\Rect(2k,2),\Rect(2k,2),\mmu\big)
\]
where $\mmu$ is any horizontal sum of $k$ Young diagrams which are either $(2,2)$ or $(4)$. Taking $\mmu_{(j)}$ to consist of $j$ copies of $(4)$ and $k-j$ copies of $(2,2)$ for $j\in \{1,2,\dots,k\}$, we have $\mmu_{(j)}=(2k+2j,2k-2j)$. Vertically summing all $k$ relations $c\big(\Rect(2k,2),\Rect(2k,2),\Rect(2k,2),\mmu_{(j)}\big)$ for each $j\in \{1,2,\dots,k\}$, which is possible since $4$ is even, we obtain
\[
c\left(\Rect(2k,2k),\Rect(2k,2k),\Rect(2k,2k),\Vsum_{j=1}^k \mmu_{(j)}\right).
\]
It is easy to see that $\Rows(\Vsum_{i=1}^k \mmu_{(j)})=2k$ as desired.
\end{proof}

\begin{figure}[h]
\begin{framed}
\[
c\begin{pmatrix}\ydiagram[*(darkgreen)]{2,2} *[*(lightgreen)]{2+2,2+2} *[*(green)]{4+2,4+2}& \ydiagram[*(darkgreen)]{2,2} *[*(lightgreen)]{2+2,2+2} *[*(green)]{4+2,4+2} & \ydiagram [*(darkgreen)]{2,2} *[*(lightgreen)]{2+2,2+2} *[*(green)]{4+2,4+2} & \ydiagram[*(darkgreen)]{4} *[*(lightgreen)]{4+4} *[*(green)]{8+4} \end{pmatrix}
\]
\[\Large \Vplus\]
\[c\begin{pmatrix}\ydiagram[*(darkred)]{2,2} *[*(lightred)]{2+2,2+2} *[*(red)]{4+2,4+2}& \ydiagram[*(darkred)]{2,2} *[*(lightred)]{2+2,2+2} *[*(red)]{4+2,4+2} & \ydiagram[*(darkred)]{2,2} *[*(lightred)]{2+2,2+2} *[*(red)]{4+2,4+2} & \ytableaushort{\none \none \none \none \none \none \none \none \none \none \none \none }*[*(darkred)]{4} *[*(lightred)]{4+4} *[*(red)]{8+2,2}
\end{pmatrix}\]
\[\Large \Vplus\]
\[c\begin{pmatrix}\ydiagram[*(darkblue)]{2,2} *[*(lightblue)]{2+2,2+2} *[*(blue)]{4+2,4+2}& \ydiagram[*(darkblue)]{2,2} *[*(lightblue)]{2+2,2+2} *[*(blue)]{4+2,4+2} & \ydiagram[*(darkblue)]{2,2} *[*(lightblue)]{2+2,2+2} *[*(blue)]{4+2,4+2} & \ytableaushort{\none \none \none \none \none \none \none \none \none \none \none \none }*[*(darkblue)]{4} *[*(lightblue)]{4+2,2} *[*(blue)]{6+2,2+2}\end{pmatrix}\]
\[\Large || \]
\[c\begin{pmatrix}\ydiagram[*(green)]{6,6,0,0,0,0} *[*(red)]{0,0,6,6,0,0} *[*(blue)]{0,0,0,0,6,6} & \ydiagram[*(green)]{6,6,0,0,0,0} *[*(red)]{0,0,6,6,0,0} *[*(blue)]{0,0,0,0,6,6}&  \ydiagram[*(green)]{6,6,0,0,0,0} *[*(red)]{0,0,6,6,0,0} *[*(blue)]{0,0,0,0,6,6}&  \ydiagram[*(green)]{12,0,0,0,0} *[*(red)]{0,10,0,0,2} *[*(blue)]{0,0,8,4,0}\end{pmatrix}\]

\caption{An illustration of Lemma~\ref{lem:squarecube} for $k=3$. The green, red, and blue each show a single $\mmu_{(j)}$ together with three $\Rect(2k,2)$ shapes. The $k$ shades of each color indicate how the relations are obtained via horizontal summation. We then vertically sum these $k$ relations (note that there are an even number of diagrams in each relation). We obtain three $\Rect(2k,2k)$ squares together with a fourth shape with $2k$ distinct row lengths, proving the lemma.}
\end{framed}
\end{figure}

\begin{lem}
\label{lem:rectcover}
Let $\llambda=\Rect(2^{4b},2^{4a})\Hplus 1_m \vdash n$ for $b\geq a\geq 1$. Then for $t\geq O\left(\frac{bn}{a2^{4a+4b}}\right)$, {the tensor power $\llambda^{\otimes t}$} covers $\Irrep(S_n)$.
\end{lem}

\begin{proof}
{Lemma~\ref{lem:need16}} implies that
\[
\Rect(2^{2a+2b},2^{2a+2b})\Hplus 1_m~\subseteq \llambda^{\otimes 2\cdot O(b/a)}.  
\]
Tensor cubing both sides and applying Lemma~\ref{lem:squarecube}, we find that $\llambda^{\otimes 2\cdot O(b/a)}$ contains {some $\mmu$ with $\Rows(\mmu)\geq2^{2a+2b}\geq 2$. The result now follows by applying Corollary~\ref{cor:powersqualitative} to $\mmu$.} (Note that the parity issues in the exponent disappear after $\Irrep(S_n)$ has been covered.)
\end{proof}

\begin{lem}
\label{lem:krectsinside}

Let $\llambda=(a_1,\dots, a_j)$ and fix $k\geq 1$. We can write
\[
	\llambda=\Vsum_{i=1}^{\lceil j/k\rceil}\left( \Rect(kh_i,k)\Hplus {\nnu_{(i)}} \right)
\]
where $\sum_{i=1}^{\lceil j/k\rceil} |{\nnu_{(i)}}|\leq k(a_1+j)$.
\end{lem}

\begin{proof}

For each $1\leq i\leq \lceil j/k\rceil$, simply take $h_i=\lfloor \frac{a_{ki}}{k}\rfloor$ and 
\[
{\nnu_{(i)}}=(a_{k(i-1)+1}-kh_i,a_{k(i-1)+2}-kh_i,\dots,a_{ki}-kh_i).
\]
Since $a_1\geq a_2\geq \dots$,
\[
\sum_{i=1}^{\lceil j/k\rceil} |{\nnu_{(i)}}| 
\leq 
\left(\sum_{i=1}^{\lceil j/k\rceil} a_{ki}-kh_{i}\right) 
+ 
k\left(\sum_{i=1}^{\lceil j/k\rceil} a_{k(i-1)+1}-a_{ki}\right).
\]
The first term on the right-hand size is a sum of at most $j$ numbers which are each at most $k$. The second term is at most $ka_1$ by telescoping. The remainder of $\llambda$ can be written as rectangles as in the lemma statement.
\end{proof}

\begin{lem}

\label{lem:finish}

Let $\sum_{j=1}^J n_j=n$. Then there exist $\llambda_{(n_j)}\vdash n_j$ such that $\llambda:=\Hsum_{j=1}^J \llambda_{(n_j)}\vdash n$ has $\Rows(\llambda)\geq \sqrt{2n}-10J$.

\end{lem}

\begin{proof}

Recalling that horizontal summation is equivalent to unioning column multisets, we form $\llambda_{(i)}\vdash n_i$ (to be horizontally summed fulfilling the lemma statement) as follows. Set $\llambda_{(1)}$ to have column lengths $1,2,\dots,a_1$ greedily until it has no more capacity, then assign the remaining column length so that $\llambda_{(1)}$ has size exactly $n_1$. Then proceed with $\llambda_{(2)}$ having columns $a_1+1,\dots,a_2$ with one final column, and similarly. At the end, suppose $\llambda:=\Hsum_{j=1}^J \llambda_{(j)}$ has $\Rows(\llambda)=m$. Then we see that $\binom{m+1}{2}+Jm\geq n$, because each of the $J$ diagrams $\llambda_{(n_j)}$ contains at most one extra column of length at most $m$. It follows that $(m+10J)^2\geq 2n$ which completes the proof. 
\end{proof}

\subsection{Case 1: $\dim(\llambda)\geq K^{n}$ for Large $K$}

Here we address the main case $\dim(\llambda)\geq K^{n}$ for $K$ a large constant. This is equivalent to $\dim(\llambda)\geq n^{\varepsilon n}$ for $\varepsilon\geq \frac{K}{\log(n)}$ (i.e. $n^{\varepsilon}$ is at least a large constant). We write the proof of this case in terms of $\varepsilon$; though not technically justified, it might be psychologically helpful to think of $\varepsilon$ as a small constant which does not go to $0$ with $n$.

Before beginning the main proof, we outline a special case. Suppose that $\llambda$ contains a macroscopic \emph{Durfee square} of side length $k=\Omega(\sqrt{n})$ (i.e. the $k$-th row of $\llambda$ has length at least $k$). Then we claim $\llambda^{\otimes t}$ covers $\Irrep(S_n)$ for $t\geq O(1)$. The reason is that we can write $\llambda=(\Rect(k,k)\Hplus \mmu)\Vplus \nnu$ and hence $\Rect(k,k)\Hplus 1_{n-k^2}\in \llambda^{\otimes 2}$. Lemma~\ref{lem:squarecube} then shows that some diagram with $\Omega(\sqrt{n})$ distinct row lengths is a subrepresentation of $\llambda^{\otimes 6}$. Finally we apply Corollary~\ref{cor:powersqualitative} to see that a small tensor power $\llambda^{O(1)}$ covers $\Irrep(S_n)$.

The full proof is a generalization of the above. We identify rectangles at different height scales and turn them into squares via Lemma~\ref{lem:rectsquare}. The hooklength formula allows us to relate the sizes of these rectangles to $\dim(\llambda)$ in general.

\begin{proof}[Proof of Theorem~\ref{thm:powersquantitative} when $\dim(\llambda)\geq K^{n}$ for Large $K$]

Write $\llambda=(a_1,a_2,\dots,a_A)$ and $\llambda'=(b_1,\dots,b_B)$. Let $H(s)$ be the $\llambda$-hooklength of a square $s\in\llambda$, defined as the number of squares directly below or directly to the right of $s$, including $s$. Let $H_r(s),H_c(s)$ be the lengths of the row and column parts of this hook, so that $H_r(s)+H_c(s)-1=H(s)$. Also for a square $s\in\llambda$ let $a(s)$ denote the length of the entire row containing $s$. Throughout we will use constant $\varepsilon_1,\varepsilon_2,\dots$ with each ratio $\frac{\varepsilon_j}{\varepsilon}$ bounded below by an absolute positive constant; each new subscript will correpond to roughly a constant factor decrease. We often omit floors and ceiling when irrelevant. Our first step is to use the hook-length formula to understand the geometry of $\llambda$. We have:
\[
n^{\varepsilon n}\leq \dim(\llambda)=\frac{n!}{\prod_{s\in\llambda}{H(s)}}\leq \frac{n^n}{\prod_{s\in\llambda}{H(s)}}=\prod_{s\in \llambda}\frac{n}{H(s)}.
\]
To control the hooklengths, we consider { the intersection $\llambda_0$ of the set of squares in $\llambda$ with a $2\sqrt{n}\times 2\sqrt{n}$ square diagram.} This leads to a decomposition 
\[
\llambda=(\llambda_R\Hplus \llambda_0)\Vplus \llambda_C=\llambda_R\Hplus (\llambda_0\Vplus \llambda_C)
\]
where $\llambda_R,\llambda_C$ respectfully consist of all columns and rows after the first $2\sqrt{n}$. {(These sets of rows and columns do not overlap because $(2\sqrt n)^2>n$.)} Observe that 
\[
n^{\varepsilon n}\leq \prod_{s\in \llambda}\frac{n}{H(s)}\leq \left(\prod_{s\in \llambda_R\Hplus\llambda_0}\frac{n}{H_r(s)}\right)\left(\prod_{s\in \llambda_C\Vplus\llambda_0}\frac{n}{H_c(s)}\right). 
\]
Assuming without loss of generality that the first product on the right side is larger, we obtain 
\begin{align*}
	n^{\varepsilon n/2}&\leq \prod_{s\in \llambda_R\Hplus\llambda_0}\frac{n}{H_r(s)} 
	\\
	\implies \varepsilon_1 n &\leq 
	\sum_{s\in\llambda_R\Hplus\llambda_0} \log_n\left(\frac{n}{H_r(s)}\right)
	\\
	&= \sum_{i=1}^{2\sqrt{n}}\log_n\left(\frac{n^{a_i}}{a_i!}\right).
\end{align*}
Using the fact $\left(\frac{a_i}{e}\right)^{a_i}\leq a_i!$ we see that 
\[
\log_n\left(\frac{n^{a_i}}{a_i!}\right) \leq a_i\log_n\left(\frac{n}{a_i}\right) + \frac{a_i}{\log(n)}. 
\]
Summing and recalling that $\varepsilon\cdot\log(n)$ is at least a large constant we obtain
\begin{align}
\nonumber
	\varepsilon_1n
	&\leq 
	\sum_{i=1}^{2\sqrt{n}} 
	\log_n\left(\frac{n^{a_i}}{a_i!}\right)
	\\
	&{\leq}
\nonumber
	\left(\sum_{i=1}^{2\sqrt n} a_i\log_n\left(\frac{n}{a_i}\right)\right)+\frac{n}{\log(n)}
	\\
	\implies 
\label{eq:HLF}
	\varepsilon_2n&\leq \sum_{i=1}^{2\sqrt{n}} a_i\log_n\left(\frac{n}{a_i}\right). 
\end{align}

Examining Equation~\eqref{eq:HLF}, we see that both large and small rows contribute a small amount to the sum. For large rows we have
\[
\sum_{i:a_i\geq n^{1-\varepsilon_3}} a_i \log_n\left(\frac{n}{a_i}\right) \leq \varepsilon_3 \sum_{i=1}^{2\sqrt n} a_i \leq \varepsilon_3 n.
\]
For small rows, note that the rows of length $a_i\leq n^{0.4}$ contribute in total at most $2n^{0.9}=o(\varepsilon n)$.

Choosing $\varepsilon_2,\varepsilon_3$ to ensure that $\varepsilon=O(\varepsilon_2-\varepsilon_3-o(\varepsilon))$, we obtain:
\begin{align}
	\varepsilon_4 n&\leq \sum_{i:n^{0.4}\leq a_i\leq n^{1-\varepsilon_3}} a_i\log_n\left(\frac{n}{a_i}\right)
	\\
	&=
	\sum_{s\in \llambda:n^{0.4}\leq a(s)\leq n^{1-\varepsilon_3}} \log_n\left(\frac{n}{a(s)}\right)\label{eq:HLF2}.
\end{align}

We partition the rows of $\llambda$ into scales according to the value $\log_n\left(\frac{n}{a_i}\right)$. Explicitly, we set $\alpha_j=\varepsilon_4\cdot 1.1^{j}$ for integers $0\leq j \leq J:=\left\lceil\log_{1.1}\left(\frac{0.6}{\varepsilon_4}\right)\right\rceil\leq O(\log\log n)$ and let $\mmu_{(j)}$ be the Young diagram consisting of all rows of $\llambda$ with lengths $a_i$ in the range $[n^{1-\alpha_{j+1}},n^{1-\alpha_j})$. 

The result is a decomposition \begin{equation}\label{eq:lambdascale}\llambda=\nnu\Vplus \Vsum_{j=0}^J\mmu_{(j)}.\end{equation} Here $\nnu$ consists of all the short and long rows not included in any $\mmu_{(j)}$, as well as the part $\llambda_C$ of $\llambda$ below the square $D$. This decomposition has the following properties:

\begin{enumerate}

\item All rows of $\mmu_{(j)}$ have length in the range $[n^{1-\alpha_{j+1}},n^{1-\alpha_j})$ for $\alpha_j= {\varepsilon_4\cdot 1.1^{j}}$.
\item From Equation~\eqref{eq:HLF2},
\[ 
	\sum_{j=0}^J \alpha_j |\mmu_{(j)}|\geq \varepsilon_5 n.
\]

\end{enumerate}

We now apply Lemma~\ref{lem:krectsinside} to each $\mmu_{(j)}$ using $k_j=16^{\lfloor \log_{16}(n^{\alpha_j/2})\rfloor}\asymp n^{\alpha_j/2} $. We get:	
\begin{equation}
\label{eq:sj}
	\mmu_{(j)}
	=
	\Vsum_{i}
	\left( \Rect(k_j h_{i,j},k_j)\Hplus \nnu_{i,j} \right).
\end{equation}
$\mmu_{(j)}$ has all row lengths at most $n^{1-\alpha_j}$, and most $2n^{1/2}$ rows. We use this to estimate the total size of the error partitions $\sum_{i}|\nnu_{i,j}|$: for each $j$,
\begin{align*}
	\sum_i |\nnu_{i,j}|&\leq k_j(n^{1-\alpha_j}+2n^{0.5})
	\\
	&\asymp n^{1-\frac{\alpha_j}{2}}+n^{\frac{1+\alpha_j}{2}}.
\end{align*}
Next we will apply the semigroup property to Equation~\eqref{eq:sj}. First, Lemma~\ref{lem:rectcube} gives 
\[
	c\left(\Rect(k_j h_{i,j},k_j),\Rect(k_j h_{i,j},k_j),\Rect(k_j \tilde h_{i,j},k_j)\Hplus 1_{k_j^2(h_{i,j}-\tilde h_{i,j})}\right)
\] 
for any $\tilde h_{i,j}\leq h_{i,j}$. We also recall that $c\big(\nnu_{i,j},\nnu_{i,j},1_{|\nnu_{i,j}|}\big)$ holds for each $i,j$. Applying the semigroup property with these Kronecker relations on Equation~\eqref{eq:sj} and setting $h_j:=\sum_i h_{i,j}$, we obtain for any $\tilde h_j\leq h_j$ and some appropriate value $r_j$:
\[
	c\left(\mmu_{(j)},\mmu_{(j)},\Rect(k_j\tilde h_j,k_j)\Hplus 1_{r_j}\right).
\]
Combining Equations~\eqref{eq:lambdascale} and \eqref{eq:sj} together with $c\big(\nnu,\nnu,1_{|\nnu|}\big)$ via the semigroup property implies, for $r:=|\nnu|+\sum_{j=0}^J r_j$, 
\begin{equation}
\label{eq:1}
	c\left(\llambda,\llambda,\left(\Hsum_{j=0}^J \Rect(k_j\tilde h_j,k_j)\right) \Hplus 1_{r}\right).
\end{equation}

We (for convenience as remarked before Lemma~\ref{lem:need16}) set $\tilde h_{j}\leq h_j$ to be the largest power of $16$ which is at most $h_j$. Recall the previous conclusion $\sum_{j=0}^J \alpha_j |\mmu_{(j)}| \geq \varepsilon_5 n$, the value $k_j\asymp n^{\alpha_j/2}$. Also note the simple estimates $k_j^2 \tilde h_j\geq \frac{|\mmu_{(j)}|-\sum_i |\nnu_{i,j}|}{100}$ and $\alpha_j\leq 0.8$ for all $j$. These together imply 
\begin{align*}
	\sum_j k_j^2\tilde h_j \log(k_j)
	&\geq 
	\Omega\left(\sum_j \alpha_j \log(n) |\mmu_{(j)}|\right) - O\left(\sum_{i,j} \alpha_j \log(n)|\nnu_{i,j}|.\right)
	\\
	&\geq 
	\Omega(\varepsilon_5 n\log(n)) 
	- 
	O\left(
	\log(n)
	\sum_{j=0}^J 
	\alpha_j\left(
	n^{1-\frac{\alpha_j}{2}}+n^{\frac{1+\alpha_j}{2}} 
	\right)
	\right).
\end{align*}
We now estimate the last term $\log(n)\sum_{j=0}^J \alpha_j\left(n^{1-\frac{\alpha_j}{2}}+n^{\frac{1+\alpha_j}{2}} \right)$. We will show that it is at most a small constant (depending on $K$) times $\varepsilon_4 n\log(n)$ so that we may ignore it in the last expression (as it is dominated by the other term). Since $\alpha_j\leq 0.8$ and $J\leq O(\log\log n)$ the contribution from all terms $n^{\frac{1+\alpha_j}{2}}$ is $O(n^{0.95})$. So we focus on upper bounding $\log(n)\sum_{j=0}^J \alpha_jn^{1-\frac{\alpha_j}{2}} = n\log(n)\sum_{j=0}^J \alpha_j n^{-\alpha_j/2}$. Recall we are in the case that $\varepsilon \log(n)$ is at least a large constant, or equivalently $n^{\varepsilon}$ at least a large constant. Hence we know that $\alpha_{j+1}-\alpha_j=\frac{\alpha_j}{10}\geq \frac{\varepsilon_4 }{10}$ is at least a large constant times $\frac{1}{\log(n)}$. This means that the sequence $(n^{-\alpha_j/2})_{j=0}^J$ is dominated by a geometrically decaying sequence with common ratio (say) $1/10$ and starting value $n^{-\alpha_0/2}=n^{-\varepsilon_4/2}$ which is at most a small constant. Because
{ $\frac{\alpha_{j+1}}{\alpha_j}=1.1$,} 
the sum $\sum_{j=0}^J \alpha_jn^{1-\frac{\alpha_j}{2}}$ is bounded above by a geometric series with common ratio $\frac{1.1}{10}\leq 1/5$, hence up to a constant factor by its first term $\varepsilon_4 n^{1-\frac{\varepsilon_4}{2}}$. Using one more time the assumption that $n^{\varepsilon}$ is at least a large constant we conclude that $\log(n)\sum_{j=0}^J \alpha_j\left(n^{1-\frac{\alpha_j}{2}}+n^{\frac{1+\alpha_j}{2}} \right)$ is at most a small constant times $\varepsilon_4 n\log(n)$. 

In summary we have established 
\[
\sum_{j=0}^J k_j^2\tilde h_j \log(k_j)\geq \varepsilon_6 n \log(n).
\]
As a result, we may choose non-negative integers $\{m_j:j\leq J\}$ such that $\sum_j m_j=r\leq n$ and 
\[
m_j\leq \frac{k_j^2 \tilde h_j\log(k_j)}{\varepsilon_7 \log(n)}.
\]
We then rewrite Equation~\eqref{eq:1} as
\begin{equation}
\label{eq:nearend}
	\Hsum_{j=0}^J
	\left(\Rect(k_j \tilde h_{j},k_j)\Hplus 1_{m_j}\right)
	\subseteq 
	\llambda^{\otimes 2}.
\end{equation}
Applying Lemma~\ref{lem:rectcover} with $n_j:=m_j+k_j^2\tilde h_j$ we see that for $t\geq t_j=O\left(\frac{\log(k_j\tilde h_j)n_j}{\log(k_j)k_j^2\tilde h_j}+1\right)$ the tensor power
\[
	\left(\Rect(k_j\tilde h_{j},k_j)\Hplus 1_{m_j}\right)^{\otimes t}
\] 
covers $\Irrep(S_{n_j})$. Since $n_j=O\left(\frac{k_j^2 \tilde h_j\log(k_j)}{\varepsilon_7 \log(n)}\right)$ we have $t_j=O\left(\frac{\log(k_j\tilde h_j)}{\varepsilon_7 \log(n)}+1\right)$. As $k_j\tilde h_j\leq n$ we conclude that $t_j=O(\varepsilon_7^{-1})$ for all $j$. By the semigroup property applied to Equation~\eqref{eq:nearend}, this means for any irreducible representations $\ggamma^{(j)}\vdash n_j$ we have 
\[
	\Hsum_{j=0}^J \ggamma^{(j)} \in \llambda^{\otimes O(\varepsilon_7^{-1})}.
\]
Lemma~\ref{lem:finish} implies that for appropriate choices of $\ggamma^{(j)}$, we have 
\[
	\Rows\left(\Hsum_{j=0}^J \ggamma^{(j)}\right)= (\sqrt{2}-o(1))\sqrt{n}.
\]
By Corollary~\ref{cor:powersqualitative} a further $O(1)$ tensor power $\llambda^{\otimes O(\varepsilon_7^{-1})}$ covers $\Irrep(S_n)$. As $\varepsilon_7=\Omega(\varepsilon)$ we conclude that $\llambda^{\otimes O(\varepsilon^{-1})}$ covers $\Irrep(S_n)$. This finishes the proof of Theorem~\ref{thm:powersquantitative} in the case that $\dim(\llambda)\geq K^n$ for large constant $K$.
\end{proof}

\subsection{Case $2$: $\dim(\llambda)\leq K^{n}$ and $\llambda$ Contains Multiple Large Rows/Columns}

Here we consider the case $\dim(\llambda)\leq K^{n}$. We first show how to break into two further cases.

\begin{lem}
If $\dim(\llambda)\leq K^{n}$, then for sufficiently large $n$ at least one row or column in $\llambda$ has size $\Omega(n)$.
\end{lem}

\begin{proof}
The hooklength formula gives 
\[
	\prod_{s\in\llambda} H(s)\geq \frac{n!}{K^n} = \Omega(n)^n.
\]
Taking logarithms on each side, we find
\begin{align*}
	\sum_{s\in\llambda} \log(H(s)) 
	&\geq 
	n\log(n)-O(n),
	\\
	\implies \frac{1}{n}
	\sum_{s\in\llambda} 
	\log\left(\frac{n}{H(s)}\right) 
	&\leq 
	O(1).
\end{align*}
Hence some hooklength has size $\Omega(n)$, which is equivalent to the desired claim.
\end{proof}

In this subsection we will prove Theorem~\ref{thm:powersquantitative} when $\dim(\llambda)\leq K^{n}$ and there are at least $2$ linear-size rows or columns. We separate this case into two subcases: $2$ long rows, or $1$ long row and $1$ long column.  (Note that the case of $2$ long columns is identical to $2$ long rows by conjugation.) In both situations we simply prove that $\llambda^{\otimes O(\log n)}$ covers $\Irrep(S_n)$ which suffices given the upper bound on $\dim(\llambda)$ (and shows $\dim(\llambda)$ is at least exponential in $n$ in Case $2$). We leave the final case of a single linear-size row to the next subsection.

\begin{lem}
\label{lem:2row}

If $\llambda=(a_1,a_2,\dots)$ for $a_1\geq a_2\geq \Omega(n)$ then $\llambda^{\otimes O(\log \, n)}$ covers $\Irrep(S_n)$.

\end{lem}

\begin{proof}

We may write 
\[
	\llambda = \big(\Rect(8\cdot 16^m,2)\Hplus \mmu\big)\Vplus \nnu
\]
for $16^m=\Omega(n)$.
(It is easy to see that whenever a Young diagram contains a rectangle we have such an equation.) By Lemma~\ref{lem:rectsquare}, $\llambda^{\otimes 4}$ contains $\Rect(16^m,16)\Hplus 1_{|\mmu|+|\nnu|}$. Directly applying {Lemma~\ref{lem:rectcover}} shows that $\Rect(16^m,16)^{\otimes {O(\log \, n)}}$ covers $\Irrep(S_{16^{m+1}})$ and in particular contains a Young diagram $\ggamma$ with $\Omega(\sqrt{n})$ distinct rows. Hence $\llambda^{\otimes O(\log \, n)}$ contains $\ggamma\Hplus 1_{|\mmu|+|\nnu|}$ which also has $\Omega({\sqrt{n}})$ distinct rows. Applying Corollary~\ref{cor:powersqualitative} shows that a further $O(1)$ tensor power suffices to cover $\Irrep(S_n)$. 
\end{proof}

The case of one long row and one long column is similar via the following Kronecker relation for hook shapes. We remark that tensor products of two hook shapes are in fact understood completely, see \cite{rosas2001kronecker}. 
 
\begin{lem}
\label{lem:hooksquare}

If $x, y\geq m$ then $c\big(\Hook(x,y),\Hook(x,y), {\Rect(2,m-1)}\Hplus 1_{x+y-2m+1}\big)$.

\end{lem}

\begin{proof}

We have $c\big(1^{m-1},1_{m-1},1^{m-1}\big)$ and $c\big(1_{m-1},1^{m-1},1^{m-1}\big)$. {(Recall that $1^{\ell}$ denotes the alternating representation of $S_{\ell}$.)} Horizontally summing yields
\[
c\left(\Hook(m,m-1),\Hook(m,m-1),{\Rect(2,m-1)}\right).
\]
Summing horizontally with $c\big(1_{x-m},1_{x-m},1_{x-m}\big)$, and then vertically in the first two arguments (and horizontally in the third) with $c\big(1^{y-m+1},1^{y-m+1},1_{y-m+1}\big)$ gives the conclusion.
\end{proof}

\begin{lem}

If $\llambda=(a_1,a_2,\dots)$ and $\llambda'=(b_1,b_2,\dots)$ for $a_1,b_1=\Omega(n)$ then then $\llambda^{\otimes O(\log \, n)}$ covers $\Irrep(S_n)$.

\end{lem}

\begin{proof}

Let $c$ be a small constant so that $a_1,b_1\geq cn$. Assuming we cannot apply Lemma~\ref{lem:2row}, we have (say) $a_2,b_2\leq cn/2$. Then it is easy to see that we may write $\llambda=\Hook(a_1,b_1-b_2)\Vplus\mmu$ for an appropriate $\mmu$. Lemma~\ref{lem:hooksquare} then implies that $\llambda^{\otimes 2}$ contains $\Rect(2,\Omega(n))\Hplus 1_r$ for appropriate $r$. By the same argument as Lemma~\ref{lem:2row} a constant tensor power of this covers $\Irrep(S_n)$, finishing the proof.
\end{proof}

\subsection{Case $3$: $\dim(\llambda)\leq K^{n}$ and $\llambda$ contains $1$ Large Row}

As usual we take $\llambda=(a_1,\dots)$ and $\llambda'=(b_1,\dots)$. Here we assume $a_1\geq cn$ and all other rows and columns have length at most $c'n$ for $c'\leq c/2$. Let $M=\max(a_2,b_1)$ and $m=a_1-M$. Write $\llambda=\hat\llambda\Hplus 1_m$. Then $\hat\llambda$ has first row of length $M=\max(a_2,b_1)$, so $|\hat\llambda|\geq {2M-1}$. The idea will be to apply a previously established case to $\hat\llambda$. Indeed, since $\hat\llambda$ has a tie for its two largest row and column lengths, we by definition cannot be in the case of a single large row. (Note that we may have $|\hat\llambda|=O(1)$, but here Proposition~\ref{prop:smallpowers} acts as a base case.) Writing 
\[
	k=n-m=|\hat\llambda|,
\]
we know $\hat\llambda^{\otimes t}$ covers $\Irrep(S_{k})$ for some $t=O\left(\frac{\log(k!)}{\log\dim(\hat\llambda)}\right)$. We next relate $\dim(\llambda)$ to $\dim(\hat\llambda)$.

\begin{figure}[h]

\begin{framed}
\[
\ydiagram[*(red)]{4,2,2,1}*[*(blue)]{4+14,0,0,0}=\ydiagram[*(red)]{4,2,2,1}\Hplus \ydiagram[*(blue)]{14,0,0,0} 
\]

\caption{In Case $3$ we break a Young diagram $\llambda$ with $1$ long row into $\hat\llambda$ (shown in red) plus a long horizontal strip (shown in blue). We ensure that $\hat\llambda$ has longest row equal to either its second longest row or longest column, so that some already-proved case of Theorem~\ref{thm:powersquantitative} applies to $\hat\llambda$. We have $|\hat\llambda|=k\geq 2M-1$, and the size of the long horizontal strip is $m=n-k$.}
\end{framed}
\end{figure}

\begin{lem}
\label{lem:SYTcombo}
With $(\llambda,\hat\llambda,n,M,k)$ as above,
\[
	\max\left(\dim(\hat\llambda),\binom{n-M}{k-M}\right)\leq \dim(\llambda)\leq \dim(\hat\llambda)\cdot\binom{n}{k}
\]
\end{lem}

\begin{proof}

We use the interpretation of dimension as counting standard Young tableaux (henceforth SYT) of a given shape. This makes it clear that $\dim(\hat\llambda)\leq \dim(\llambda)$, since any SYT of shape $\hat\llambda$ extends to a SYT of shape $\llambda$. To see that $\binom{n-M}{k-M}\leq\dim(\llambda)$ we explicitly construct $\binom{n-M}{k-M}$ SYTs of shape $\llambda$ at follows. Fill the leftmost $M$ elements of the top row with $1,2,\dots,M$. Then pick $k-M$ of the remaining $n-M$ numbers in $[n]$ to complete some SYT of shape $\hat\llambda$ inside $\llambda$ and use the remaining $n-k$ numbers to fill the top row of $\llambda$. Since $a_2\leq M$ this is a valid SYT.

To show the upper bound $\dim(\llambda)\leq \dim(\hat\llambda)\cdot\binom{n}{k}$, we argue similarly. The point is that each choice of which $k$ numbers in $[n]$ are used to label $\hat\llambda\subseteq\llambda$, combined with a choice of SYT on $\hat\llambda$ to determine their relative order, determines at most $1$ SYT of shape $\llambda$ as the remaining $m$ squares of the first row must be in sorted order.
\end{proof}

\begin{cor}
\label{cor:dimllambda}
We have
\[
\log(\dim(\llambda)) \asymp \max\left(\log(\dim(\hat\llambda)),k\log(n/k)\right).
\]
\end{cor}

\begin{proof}

Recall the simple estimate $\left(\frac{a}{b}\right)^b\leq \binom{a}{b}\leq \left(\frac{ae}{b}\right)^b$, which implies
\[
b\log(a/b)\leq \log\binom{a}{b}\leq b\left(\log(a/b)+1\right).
\]
From before we have the inequalities $c'n\geq M$ for a small absolute constant $c'$, and $(1-c')n\geq k\geq {2M-1}$. Then it is easy to see that $\binom{n-M}{k-M}\asymp \binom{n}{k}\asymp k\log\left(\frac{n}{k}\right)$. Indeed, we have 
\[
\log\frac{n-M}{k-M}\geq \log\frac{n}{k}\geq \Omega(1)
\]
and also
\[
k-M\asymp k.
\]
Combined with Lemma~\ref{lem:SYTcombo} this implies the claim.
\end{proof}

In light of the above corollary, it suffices to show that $\llambda^{\otimes t}$ covers $\Irrep(S_n)$ for 
{
\[
	t
	\geq 
	O\left(
	\frac{n\log n}{\log\dim(\llambda)}
	\right)
	= 
	O\left(
	\min\left(
	\frac{n\log n}{\log\dim(\hat\llambda)},\frac{n\log n}{k\log(n/k)}
	\right)
	\right).
\]
}
We will show separately that either of the numbers on the right hand side suffices, beginning with the first. As mentioned before, one of the previous cases of Theorem~\ref{thm:powersquantitative} applies to $\hat\llambda$ because clearly the present case 3 does not.  Therefore we have that $\hat\llambda^{\otimes t}$ covers $\Irrep(S_k)$ for $t=O\left(\frac{k \log k}{\log\dim(\hat\llambda)}\right)$. This implies $\ttau_n^{\otimes k}\subseteq \llambda^{\otimes t}$ for the same range of $t$, and hence taking $O(n/k)$-th tensor powers we see that $\llambda^{\otimes O(nt/k)}$ covers $\Irrep(S_n)$ in the same range of $t$. Since $k\leq n$ we have $\frac{n}{k}\cdot \frac{k \log k}{\log\dim(\hat\llambda)}\leq \frac{n \log n}{\log\dim(\hat\llambda)}$, so we conclude that $\llambda^{\otimes t}$ covers $\Irrep(S_n)$ for $t\geq O\left(\frac{n\log n}{\log\dim(\hat\llambda)}\right)$.

It remains to show that $\llambda^{\otimes t}$ covers $\Irrep(S_n)$ for $t\geq O\left(\frac{n\log n}{k\log(n/k)}\right)$. To do this we will essentially reduce to the cases that $\llambda$ is a hook or contains two rows. However the proofs in these cases are slightly more involved than those of the previous subsection (since the first row can have length $n-o(n)$) and use a version of the Pieri rule. We first note again that for $k=O(1)$ of constant size, the result is clear. In this case, $\hat\llambda^{\otimes O(1)}$ covers $\Irrep(S_k)$ which implies as before that $\llambda^{\otimes O(n/k)}$ covers $\Irrep(S_n)$. And $n/k\asymp \frac{n\log(n)}{k\log(n/k)}$ when $n$ has superconstant size (which we already assumed) and $k$ has constant size. Therefore we may assume that $k$ is sufficiently large.

\begin{lem}
\label{lem:hookidempotent}
For any $a,b$ we have the Kronecker relation 
\[
{	
c\big(\Hook(a,b),\Hook(a,b),\Hook\big(\max(a,b),\min(a,b)\big)\big).
}
\]

\end{lem}

\begin{proof}
By conjugating we may assume $a\geq b$. {Lemma~\ref{lem:tensorcube} implies $c\big(\Hook(b,b),\Hook(b,b),\Hook(b,b)\big)$} since $\Hook(b,b)$ is symmetric. Horizontal summation with $c\big(1_{a-b},1_{a-b},1_{a-b}\big)$ completes the proof.
\end{proof}

\begin{lem}
\label{lem:hookortworow}
For any $\llambda$ as in this subsection, $\llambda^{\otimes 2}$ contains at least one of the following:
\begin{enumerate}
	\item The two row partition $(n-\ell,\ell)$ for some $\ell=\Omega(k)$.
	\item The hook partition $\Hook(n-\ell+1,\ell)$ for some $\ell=\Omega(k)$.
\end{enumerate}

\end{lem}

\begin{proof}

First suppose that $b_1=\Omega(k)$. Then we may write $\llambda=\Hook(c'n,b_1)\Hplus \mmu$ for some $\mmu$. We use Lemma~\ref{lem:hookidempotent} and apply the relation $c\big(\mmu,\mmu,1_{|\mmu|}\big)$ and the semigroup property. As $b_1,c'n$ are both at least $\Omega(k)$ this shows we obtain a hook of the form $\Hook(n-\ell+1,\ell)$ for $\ell=\Omega(k)$ inside $\llambda^{\otimes 2}$ in this case.

Next suppose that $b_1\leq k/10$, and suppose further that $M=b_1\leq k/10$. Then we use Lemma~\ref{lem:krectsinside} on $\hat\llambda$, with the value of $k$ in that Lemma equal to $2$. This says we can write
\[
\hat\llambda=\Vsum_{1\leq i\leq \lceil b_1/2\rceil}\left( \Rect(2h_i,2)\Hplus {\nnu_{(i)}} \right)
\]
where $\sum_i |{\nnu_{(i)}}|\leq 2(M+b_1)\leq k/2$. Therefore the semigroup property (with vertical/vertical/horizontal summation in the outer layer) and the relations 
\[
	c\big({\nnu_{(i)}},{\nnu_{(i)}},{1_{|{\nnu_{(i)}}|}}\big),
	\quad 
	{
	c\big(\Rect(2h_i,2),\Rect(2h_i,2),\Rect(2h_i,2)\big)
	}
\] 
imply $c\big(\hat\llambda,\hat\llambda,\Rect(2\sum_i h_i,2)\Hplus 1_{\sum_i |{\nnu_{(i)}}|}\big)$. Since $\sum_i|{\nnu_{(i)}}|\leq k/2$ the third argument equals $\Rect(\ell,2)\Hplus 1_{k-\ell}$ for $\ell=\Omega(k)$. Applying the semigroup property again to $c\big(1_{m},1_m,1_m\big)$ and recalling $\llambda=\hat\llambda\Hplus 1_m$ gives a suitable two row partition inside $\llambda^{\otimes 2}$ in this subcase.

We also have the remaining subcase that $b_1\leq k/10$ and $M=a_2$. In this case we similarly apply Lemma~\ref{lem:krectsinside} to the 3rd row and below in $\hat\llambda$ and separately to the first two rows; the fact that the first two rows of $\hat\llambda$ have equal length improves the bound on $\sum_i|{\nnu_{(i)}}|$ since it means $|\nnu_{(1)}|\leq 2$. The conclusion of this is the same decomposition 
\[
\hat\llambda=\Vsum_{1\leq i\leq \lceil b_1/2\rceil}\left( \Rect(2h_i,2)\Hplus {\nnu_{(i)}} \right)
\]
but with the guarantee $\sum_i |{\nnu_{(i)}}|\leq 2+2(a_3+b_1)$. Since $b_1\leq k/10$ and $a_3{\leq} \frac{M+a_2+a_3}{3}\leq k/3$, it follows that $\sum_i |{\nnu_{(i)}}|\leq 2+k\cdot 0.9$. Since we assumed $k$ is at least a large constant, we again have $\sum_i h_i=k-\sum_i |{\nnu_{(i)}}|=\Omega(k)$ and similarly obtain a suitable two row partition inside $\llambda^{\otimes 2}$. This concludes all the cases of the lemma, finishing the proof.
\end{proof}

Now we finish by proving the result for hooks and two-row partitions. The key is the following version of the more general Pieri rule which we obtain directly from the semigroup property. It says that tensoring with a two-row partition allows us to move a horizontal strip down from the top row, and that tensoring with a hook allows us to extract a vertical strip from the top row.

\begin{lem}
\label{lem:pieri}
For any $\mmu\vdash (n-k)$ we have
\[
c\big((n-k,k),\mmu\Hplus 1_{k},\mmu\Vplus 1_k\big)
\]
and 
\[
c\big(\Hook(n-k+1,k),\mmu\Hplus 1_{k},\mmu\Hplus 1^k\big)
\]
\end{lem}

\begin{proof}

The first follows from the relations $c\big(1_{n-k},\mmu,\mmu\big)$ and $c\big(1_k,1_k,1_k\big)$ and the semigroup property, where we sum vertically in the first and third entries. The second follows from the relations $c\big(1_{n-k},\mmu,\mmu\big)$ and $c\big(1^k,1_k,1^k\big)$ where we sum horizontally in all entries. 
\end{proof}

\begin{figure}[h]
\begin{framed}

\[
c\begin{pmatrix}\ydiagram{6,0,0}\\ \ydiagram{3,2,1,0,0} \\ \ydiagram{3,2,1}\end{pmatrix}\Hplus c\begin{pmatrix}\ydiagram{1,1,1,1,0}\\ \ydiagram{4,0} \\ \ydiagram{1,1,1,1}\end{pmatrix}=c\begin{pmatrix}\ydiagram{7,1,1,1,0}\\ \ydiagram{7,2,1,0}\\ \ydiagram{4,3,2,1}\end{pmatrix}
\]

\caption{An illustration of the second statement in Lemma~\ref{lem:pieri} when $\mmu=\vvarrho_3$ and $k=4$}

\end{framed}
\end{figure}

\begin{lem}
\label{lem:tworowfinish}
Let $\llambda=(n-k,k)$. Then $\llambda^{\otimes t}$ covers $\Irrep(S_n)$ for $t\geq O\left(\frac{n\log n}{k\log(n/k)}\right)$.

\end{lem}

\begin{proof}

First suppose that $k^2 \leq n/10$. Then iterating the first part of Lemma~\ref{lem:pieri} $k$ times shows that $\Rect(k,k)\Hplus 1_{n-k^2}\in \llambda^{\otimes k}$. We have previously seem that $\Rect(k,k)^{\otimes O(1)}$ covers $\Irrep(S_{k^2})$ and that a further $O(n/k^2)$ tensor power then suffices to cover all of $\Irrep(S_n)$. In total this shows that $\llambda^{\otimes O(n/k)}$ covers $\Irrep(S_n)$ as claimed. (Note that if $k^2\leq n$ then $\log(n)\asymp \log(n/k)$.)

Next suppose that $k^2\geq n/10$. Then iterating the first statement of Lemma~\ref{lem:pieri} shows that $\Rect(k,h)\Hplus 1_{n-hk}\in\llambda^{\otimes h}$ for $h=\Theta(n/k)\leq k$. Here we have $h\geq 2$, and $k\geq 1000$ (as we already handled the case where $k=O(1)$.) We essentially will just apply Lemma~\ref{lem:rectcover} but we technically need to ensure $k,h$ are powers of $16$.

We first assume $h\geq 16$. Then we simply round $h,k$ down to $\tilde h,\tilde k$, the largest smaller powers of $16$, and observe that 
\begin{equation}
\label{eq:Rectkh}
\Rect(k,h)\Hplus 1_{n-hk}=(\Rect(\tilde k,\tilde h)\Vplus \nnu)\Hplus \mmu 
\end{equation}
for appropriate $\mmu,\nnu$. As we assumed $h\leq k$ we have that $\tilde h$ divides $\tilde k$ and so {by Lemma~\ref{lem:rectcube},}
\begin{equation}
\label{eq:Recttildekh}
	c\big(\Rect(\tilde k,\tilde h),\Rect(\tilde k,\tilde h),\Rect(\tilde k,\tilde h)\big).
\end{equation}
{Combining \eqref{eq:Rectkh} and \eqref{eq:Recttildekh}} using the semigroup property gives 
\begin{align*}
	&c\big(\Rect(k,h)\Hplus 1_{n-hk},\Rect(k,h)\Hplus 1_{n-hk},\Rect(\tilde k,\tilde h)\Hplus 1_{n-\tilde k\tilde h}\big)
	\\
	\implies 
	&\Rect(\tilde k,\tilde h)\Hplus 1_{n-\tilde k\tilde h}\in {\llambda^{\otimes O(n/k)}}.
\end{align*}
Now we can apply Lemma~\ref{lem:rectcover} to finish. It shows that a further $O\left(1+\frac{n\log k}{n\log (n/k)}\right)$ tensor powers are needed. As $k^2\geq n/10$ this is exactly what we wanted to show.

If $h\leq 15$ we act similarly, with $\tilde h=2$ and $\tilde k$ the largest number which is $8$ times a power of $16$ and at most $k$. Using $\Rect(a,16)\in \Rect(4a,4)^{\otimes 2}\subseteq \Rect(8a,2)^{\otimes 4}$ we obtain rectangles with the same size up to constants whose side lengths are powers of $16$ and proceed identically to the above.
\end{proof}

\begin{lem}
\label{lem:hookfinish}
Let $\llambda=\Hook(n-k+1,k)$. Then $\llambda^{\otimes t}$ covers $\Irrep(S_n)$ for $t\geq O(\frac{n\log n}{k\log(n/k)})$.

\end{lem}

\begin{proof}

The proof is identical to the two-row case above using instead the second statement of Lemma~\ref{lem:pieri} and conjugating all rectangles.
\end{proof}

Combining Lemmas~\ref{lem:hookortworow}, \ref{lem:tworowfinish}, \ref{lem:hookfinish} we conclude that $\llambda^{\otimes t}$ covers $\Irrep(S_n)$ for $t\geq O(\frac{n\log n}{k\log(n/k)})$. This completes the proof of Theorem~\ref{thm:powersquantitative} in case 3 and hence in general.

\section*{Acknowledgement}

The author gratefully acknowledges support of NSF and Stanford Graduate Fellowships. I thank Daniel Bump, Pavel Etingof, Xiaoyue Gong, Sammy Luo, Alex Malcolm, Chris Ryba and the anonymous referee for helpful comments, corrections, discussions, and references. I thank Greta Panova for bringing \cite{liebeck2019diameters} to my attention after the initial posting of this paper.

\bibliographystyle{alpha}
\bibliography{randomsaxl}

\appendix

\section{Alternate Proof of Fourth Power Saxl Theorem}

Here we give an alternate proof of Theorem~\ref{thm:SaxlFourth} that $\vvarrho_r^{\otimes 4}$ covers $\Irrep(S_n)$ for $r$ sufficiently large (which is Theorem 1.4 of \cite{LuoSellke}) based on another main result from \cite{LuoSellke}. The implication is immediate from a lemma on the representation theory of an arbitrary finite group $G$ which we suspect to be known but have not been able to locate in the literature. 

\begin{defn}

Let $G$ be a finite group and let $M_G$ be the Plancherel probability measure on $\Irrep(G)$ which assigns an irreducible representation $\llambda$ a probability $M_G(\llambda)=\frac{\dim(\llambda)^2}{|G|}$. For an arbitrary finite-dimensional $G$-representation $V$, let $M_G(V)$ denote the Plancherel measure of the set of distinct irreducible subrepresentations of $V$.

\end{defn}

Theorem 1.6 of \cite{LuoSellke} states that $\vvarrho_r^{\otimes 2}$ contains Plancherel-asymptotically-almost-all of $\mathcal Y_n$ for $n=\binom{r+1}{2}$, i.e. $\lim_{r\to\infty} M_{S_n}(\vvarrho_r^{\otimes 2})=1$. Therefore Theorem~\ref{thm:SaxlFourth} follows immediately from the lemma below. We note that proving Theorem 1.6 of \cite{LuoSellke} relies on the deep work of \cite{Borodin}, so the proof of Theorem~\ref{thm:SaxlFourth} given in \cite{LuoSellke} is more elementary than the present proof. Nonetheless we find the connection enlightening.

\begin{lem}
\label{lem:appendix}
Suppose $M_G(V)+M_G(W)>1$. Then $V\otimes W$ covers $\Irrep(G)$.

\end{lem}

\begin{proof}

The conclusion is equivalent to the statement that $\langle \chi^V\chi^W,\chi^{\llambda^*}\rangle >0$ for any irreducible representation $\llambda$, where $(\cdot)^{*}$ denotes the dual representation. As $\langle \chi^V\chi^W,\chi^{\llambda^*}\rangle=\langle \chi^W,\chi^{V^*}\chi^{\llambda^*}\rangle$, this is equivalent to showing the tensor product $V^*\otimes \llambda^*$ shares some irreducible subrepresentation with $W$. We will prove that $M_G(V\otimes \llambda)\geq M_G(V)$. This implies $M_G(V^*\otimes \llambda^*)+M_G(W)=M_G(V\otimes \llambda)+M_G(W)>1$ so that they share a subrepresentation by the pigeonhole principle.

To see this, we work in the standard inner product space $L^2(G)$ and recall that irreducible characters $\chi^{\mmu}$ for $\mmu\in\Irrep(G)$ are orthonormal. We identify representations with their characters. Consider for any representation $U$ the best $L^2$ approximation to the regular representation $\Reg$ of $G$ lying in the linear space 
\[
	\left\{\sum_{\mmu_U\in U} a_{\mmu_U}{\chi^{\mmu_U}}~:~a_{\mmu_U}\in \mathbb R\right\}.
\]
From the point of view of irreducible representations it is clear that the best approximation $\tilde U$ is obtained by projection via $a_{\mmu_U}=\dim(\mmu_U)$, and the $L^2$ error of this approximation $\tilde U$ is therefore $|\Reg-\tilde U|=\sqrt{|G|\cdot (1-M_G(U))}$.

We form $\tilde V$ and multiply its character by $\frac{\chi^{\llambda}}{\dim\llambda}$, and by abuse of notation treat this as a tensor product of fractional representations. The key point is that $\tilde V\otimes \frac{\llambda}{\dim\llambda}$ has the same character value at the identity element of $G$, and a smaller character value (in absolute value) at all other elements. Since $\Reg$ has character value $0$ at all non-identity elements, computing the distances using the character basis implies
\[
\left|\tilde V\otimes \frac{\llambda}{\dim\llambda}-\Reg\right|\leq |\tilde V-\Reg|.
\]
Moreover $\tilde V\otimes \frac{\llambda}{\dim\llambda}$ is in the $\mathbb R$-span of the irreducible subrepresentations of $V\otimes \llambda$. Since the function $\sqrt{1-M_G(U)}$ is decreasing in $M_G(U)$, the fact that by using subrepresentations of $V\otimes \llambda$ we weakly improved upon the best $L^2$ approximation to $\Reg$ using subrepresentations of $V$ implies $M_G(V\otimes \llambda)\geq M_G(V)$ as desired.
\end{proof}

Note that Lemma~\ref{lem:appendix} becomes completely false if Plancherel measure is replaced by uniform measure. For instance, the group of invertible affine transformations of $\mathbb F_p$ has $p-1$ irreducible representations of dimension $1$ and one of dimension $p-1$. If $V,W$ each contain exactly the $1$-dimensional irreducible representations then $V\otimes W$ still consists of only $1$-dimensional irreducibles, hence does not cover $\Irrep(G)$.

\end{document}